\documentclass[reqno,oneside,12pt]{amsart}
\usepackage{
  amsmath,
  amssymb,
  amsthm,
  geometry,
  paralist,
  thmtools,
  tikz,
  tikz-cd,
  verbatim,
  relsize, 
  bbm, 
  dsfont,
  mathtools
}

\usepackage[T1]{fontenc}

\usepackage{caption}
\usepackage{subcaption}
\usetikzlibrary{angles}

\renewcommand\labelenumi{(\roman{enumi})}
\renewcommand\theenumi\labelenumi

\usepackage{tcolorbox}

\RequirePackage{ifpdf}
\ifpdf
  \IfFileExists{pdfsync.sty}{\RequirePackage{pdfsync}}{}
  \RequirePackage[pdftex,
   colorlinks = true,
   urlcolor = myblue, 
   citecolor = mygreen, 
   linkcolor = myred, 
   bookmarksopen=true,
   unicode, psdextra]{hyperref}
\else
  \RequirePackage[hypertex, unicode, psdextra]{hyperref}
\fi
\hypersetup{
    colorlinks=true,
    linkcolor=blue,
    filecolor=magenta,      
    urlcolor=cyan,
    citecolor = blue
} 

\RequirePackage{ae, aecompl, aeguill} 

\usepackage{quiver}

\usepackage{cleveref} 

\usepackage{adjustbox}

\tikzcdset{
	arrows = crossing over
}

\usepackage[maxbibnames=99,
			maxalphanames=10, 
			style=alphabetic]{biblatex}
\usepackage{xurl}

\usepackage{dynkin-diagrams}

\usepackage[boxsize=.5em]{ytableau}

\usetikzlibrary{arrows.meta} 

\usepackage[textsize=scriptsize]{todonotes}
\presetkeys{todonotes}{inline}{}

\makeatletter
\providecommand\@dotsep{5}
\makeatother

\newcommand{\bbB}{\mathbb B}

\newcommand{\bbR}{\mathbb R}

\newcommand{\bbW}{\mathbb W}
\newcommand{\bbZ}{\mathbb Z}

\newcommand{\<}{\langle}
\renewcommand{\>}{\rangle}

\DeclareMathOperator{\id}{id}

\DeclareMathOperator{\sgn}{sgn}
\DeclareMathOperator{\even}{even}
\DeclareMathOperator{\FPdim}{FPdim}
\DeclareMathOperator{\reg}{reg}

\DeclareMathOperator{\ra}{\rightarrow}

\DeclareMathOperator{\Hom}{Hom}

\declaretheorem[numberwithin=section]{theorem}
\declaretheorem[sibling=theorem]{lemma}
\declaretheorem[sibling=theorem]{corollary}
\declaretheorem[sibling=theorem]{conjecture}

\declaretheorem[sibling=theorem]{proposition}

\declaretheorem[sibling=theorem, style=remark]{remark}
\declaretheorem[sibling=theorem, style=definition]{definition}
\declaretheorem[sibling=theorem, style=definition]{example}

\newtheorem*{theorem*}{Theorem}

\crefname{lemma}{Lemma}{Lemma}
  \crefname{corollary}{Corollary}{Corollary}
  \crefname{theorem}{Theorem}{Theorem}
  \crefname{definition}{Definition}{Definition}
   \crefname{proposition}{Proposition}{Proposition}
\crefname{question}{Question}{Question}
 \crefname{section}{Section}{Section} 
   \crefname{construction}{Construction}{Construction}
   \crefname{generalization}{Generalization}{Generalization}
  \crefname{construction}{Construction}{Construction}
  \crefname{notation}{Notation}{Notation}
   \crefname{example}{Example}{Example}
  \crefname{remark}{Remark}{Remark}
  \crefname{fact}{Fact}{Fact}
  \crefname{conjecture}{Conjecture}{Conjecture}
  \crefname{motivation}{Motivation}{Motivation}  
  \crefname{figure}{Figure}{Figure}  
  \crefname{assumption}{Assumption}{Assumption}

\begin{document}

\title[]{Representations of Coxeter groups over fusion rings and hyperplane complements}

\author[]{Edmund Heng}
\address{School of Mathematics and Statistics, University of Sydney, Australia}
\email{edmund.heng@sydney.edu.au}

\author[]{Luis Paris}
\address{Universit\'{e} Bourgogne Europe, CNRS, IMB, UMR 5584, 21000 Dijon, France}
\email{lparis@u-bourgogne.fr}



\begin{abstract}
	We study faithful realisations of Coxeter groups over fusion rings and study Vinberg systems associated to them. 
    We show that they induce embeddings of hyperplane complements, which provide geometrical realisations of certain types of strong admissible (LCM) homomorphisms between Artin--Tits groups.
\end{abstract}

\maketitle

\tableofcontents

\section{Introduction}
\subsection{Hyperplane systems and Coxeter groups}
The geometric representation of a Coxeter group (also known as the Tits representation) is a fundamental tool in understanding Coxeter groups.
One of the many properties it provides us with is the realisation of Coxeter groups as real reflection groups in the sense of Vinberg \cite{Vinberg_reflectiongroups}, obtained through the associated contragredient representation.
Associated to this is a hyperplane system in the (open) Tits cone, where each chamber is labelled by an element of the Coxeter group, and each hyperplane corresponds to a reflection in the Coxeter group.
Not only does this give us a nice geometrical understanding of Coxeter groups, but the corresponding complexified hyperplane complement is the key object that relates Coxeter groups and Artin--Tits groups (also known as generalised braid groups) \cite{VdL_thesis}.

In this paper, we also study hyperplane systems associated with Coxeter groups, where the starting point is instead representations over a special family of rings known as \emph{fusion rings}.
The upshot of this consideration will be that these provide us with nice embeddings of hyperplane complements associated to different Coxeter groups that are related via LCM foldings \cite{Crisp99} and strong admissible foldings \cite{Muhlherr_CoxinCox,Castella_admissible}.

\subsection{Fusion rings}
Fusion rings are certain (unital, associative, but possibly non-commutative) rings that satisfy nice integrality and positivity structures (see \cref{defn:fusionring}).
They are central in the study of a variety of different subjects such as subfactor theory and conformal field theory; they also arise as Grothendieck rings of fusion categories, which are used to construct Turaev--Viro invariants for knots and 3-manifolds.

Our work stems from recent papers by the first-named author and co-authors \cite{heng2023coxeter,EH_fusionquivers,HL24}.
Briefly, while representations of quivers induce representations of Weyl groups defined over $\bbZ$, certain generalised representations of quivers in fusion categories extend this connection beyond Weyl groups to Coxeter groups -- i.e.\ including non-crystallographic types such as $H_3, H_4$ and $I_2(m)$.
Here, the representations of Coxeter groups obtained are instead defined over fusion rings $R$, which nonetheless satisfy similar properties as geometric representations over $\bbR$ -- we call these \emph{geometric representations over $R$}; see \cref{defn:geometricfusionrep} and cf.\ \cref{defn:geometricrealrep}.

\subsection{Main results}
As with geometric representations over $\bbR$, each geometric representation over a fusion ring $R$ induces a contragredient \emph{real representation} $\Theta$, on which $\bbW$ acts as a real reflection group in the sense of Vinberg; see \cref{sec:contragredientfusionrep}.
In fact, $\Theta$ is canonically isomorphic to the contragredient representation of a geometric $\bbR$-representation of $\bbW$.
The novel aspect comes from the integral structure of $R$, which allows us to associate a (typically larger) real vector space $\check{\Theta}$ containing $\Theta$, where $\bbW$ continues to act $\bbR$-linearly, albeit not necessarily as a reflection group (each generator acts as a product of reflections in $\check{\Theta}$).
Instead, there is an ``unfolded'' Coxeter group $\check{\bbW}$ that acts on $\check{\Theta}$ as a real reflection group, and the inclusion $\Theta \subseteq \check{\Theta}$ is $\bbW$-equivariant with respect to an injective group homomorphism $\varphi: \bbW \to \check{\bbW}$; see \cref{sec:unfoldcontragredientfusionrep} and \ref{sec:fusionreptoZrep}.\footnote{Both $\check{\Theta}$ and $\check{\bbW}$ depend on the geometric representation over $R$; see \cref{rmk:distinct_realisation}.}

All in all, each geometric $R$-representation provides \emph{two} hyperplane systems.
On one hand, we have the hyperplane system $(T^\circ, \mathcal{A})$ associated to the action of $\bbW$ on $\Theta$; on the other, we have the hyperplane system $(\check{T}^\circ, \check{\mathcal{A}})$ associated to the action of $\check{\bbW}$ on $\check{\Theta}$.
Here, $T^\circ \subseteq \Theta$ and $\check{T}^\circ \subseteq \check{\Theta}$ denote the respective Tits cones, whereas $\mathcal{A}$ and $\check{\mathcal{A}}$ denote the respective sets of hyperplanes.
Our first result shows that the embedding $\Theta \subseteq \check{\Theta}$ restricts to an embedding of their respective real hyperplane complements:
\begin{theorem}[=\cref{thm:hyperplanesystem}]
    The set of hyperplanes $\mathcal{A}$ can be obtained from $\check{\mathcal{A}}$ as follows:
    \[
    \mathcal{A} = \{ \check{H} \cap \Theta \mid \check H \in \check{\mathcal A} \text{ and } \check H \cap T^\circ \neq \emptyset \}.
    \]
    Moreover, $T^\circ \subseteq \check{T}^\circ$, and so the inclusion $\Theta\subseteq \check{\Theta}$ restricts to an embedding of hyperplane complements
    \[
    T^\circ \setminus \bigcup_{H \in \mathcal{A}} H \subseteq \check{T}^\circ \setminus \bigcup_{\check{H} \in \check{\mathcal{A}}} \check{H}.
    \]
\end{theorem}

In particular, every hyperplane $\check{H} \in \mathcal{\check{A}}$ intersects with $\Theta$ to give a hyperplane $\check{H}\cap \Theta = H \in \mathcal{A}$ -- unless $\check{H} \cap \Theta$ lives outside of the Tits cone $T^\circ$.
Note that if $\bbW$ is finite, then necessarily $T^\circ = \Theta$ and hence all the hyperplanes in $\check{\mathcal{A}}$ intersect $\Theta$ to give the hyperplanes in $\mathcal{A}$.
An example of an embedding obtainable from such consideration is given in \cref{fig:pentagonunfold}, where  $\bbW \cong D_5$ (type $I_2(5)$) and $\check{\bbW} \cong S_5$ (type $A_4$).

\begin{figure}[h]
\footnotesize
\begin{tikzpicture}
\def\R{1.5}

\def\L{1.5}

\draw[thick]
    (0:\R) -- (72:\R) -- (144:\R) -- (216:\R) -- (288:\R) -- cycle;

\foreach [count=\x] \angle in {0, 72, 144, 216, 288} {
	\node (p\x) at (\angle:{\R*\L}) {};
	\node (p'\x) at (180+\angle:{\R*\L}) {};
}

\foreach \x in {1,...,5} {
	\draw[red] (p\x) -- (p'\x);
}

\node[right] at (p1) {$\check{H}_{\alpha_{1}} \cap \check{H}_{\alpha_{3}}$};
\node[right] at (p'4) {$\check{H}_{\alpha_{4}} \cap \check{H}_{\alpha_{2}}$};
\node[right] at (p2) {$\check{H}_{\alpha_{1} + \alpha_{2}} \cap \check{H}_{\alpha_{2} + \alpha_{3} + \alpha_{4}}$};
\node[right] at (p5) {$\check{H}_{\alpha_{3} + \alpha_{2}} \cap \check{H}_{\alpha_{1} + \alpha_{2} + \alpha_{3} + \alpha_{4}}$};
\node[right] at (p'3) {$\check{H}_{\alpha_{3} + \alpha_{4}} \cap \check{H}_{\alpha_{1} + \alpha_{2} + \alpha_{3}}$};

\end{tikzpicture}
\caption{\footnotesize The hyperplane complement $T^\circ \setminus \bigcup_{H \in \mathcal{A}} H = \Theta \setminus \bigcup_{H \in \mathcal{A}} H$ for $\bbW\cong D_5$, with hyperplanes $\check{H}_? \in \check{\mathcal{A}}$ associated to $\check{\bbW} \cong S_5$ indicated. The 10 hyperplanes for $\check{\bbW}$ (labelled by the positive roots) intersect in pairs to give the 5 hyperplanes (lines) for $\bbW$ in $\Theta$.}
\label{fig:pentagonunfold}
\end{figure}
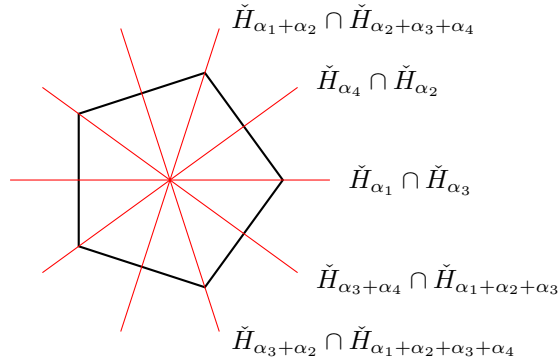

Now let $\Omega_{\reg}$ and $\check{\Omega}_{\reg}$ be the corresponding complexified hyperplane complements associated to $(T^\circ, \mathcal{A})$ and $(\check{T}^\circ, \check{\mathcal{A}})$ respectively.
By the result of Van der Lek \cite{VdL_thesis}, the fundamental groups $\pi_1(\Omega_{\reg}/\bbW)$ and $\pi_1(\check{\Omega}_{\reg}/\check{\bbW})$ are the Artin--Tits groups of $\bbW$ and $\check{\bbW}$ respectively.
\begin{theorem}[=\cref{thm:hyperplanecompembedding}]
    The embedding $\Theta \subseteq \check{\Theta}$ induces an embedding on the complexified hyperplane complements $\Omega_{\reg} \subseteq \check{\Omega}_{\reg}$.
    The induced map on fundamental groups $\tilde{\varphi}: \pi_1(\Omega_{\reg}/\bbW) \to \pi_1(\check{\Omega}_{\reg}/\check{\bbW})$ lifts the homomorphism $\varphi$; i.e.\ the following diagram commutes:
    \[
    \begin{tikzcd}
        \pi_1(\Omega_{\reg}/\bbW) \ar[r,"\tilde{\varphi}"] \ar[d, two heads] & 
            \pi_1(\check{\Omega}_{\reg}/\check{\bbW}) \ar[d, two heads] \\
        \bbW \ar[r,"\varphi"] &
            \check{\bbW}.
    \end{tikzcd}
    \]
\end{theorem}
The maps $\varphi$ and $\tilde{\varphi}$ are expected to be induced from a strong admissible partition on the Coxeter graph of $\check{\bbW}$ in the sense of \cite{Muhlherr_CoxinCox,Castella_admissible}, which generalises LCM foldings studied in \cite{Crisp99}.
In the case of categorifiable fusion rings $R$, this is a consequence of results in \cite{EH_fusionquivers}; see \cref{thm:unfoldisLusztig}.
Moreover, one may choose fusion rings $R$ so that $\varphi$ and $\tilde{\varphi}$ agree with (a variant of) LCM homomorphisms used in \cite{CrispParis_2001} to prove the Tits conjecture, and in \cite{Paris_monoidinject} to prove that the Artin--Tits monoid injects into its group; see \cref{sec:partitionfoldingLCM}.


\subsection*{Acknowledgements}
EH would like to thank the IH\'ES, where part of this work was carried out. EH would also like to thank Universit\'e Bourgogne Europe for the research visit, which resulted in this collaboration -- special thanks goes to Federica Gavazzi who made this visit possible.

\section{Preliminaries}
\subsection{Coxeter systems and realisations}
We begin by recalling the definition of a Coxeter system and its associated Artin--Tits group.
\begin{definition} \label{defn:Coxsystem}
Let $S$ be a finite set.
A \emph{Coxeter matrix} over $S$ is a symmetric square matrix $(m_{s,t})_{s,t \in S}$ with diagonal $m_{s,s} = 1$ for all $s\in S$ and $m_{s,t} =m_{t,s} \in \{2,3,..., \infty\}$ for $s \neq t$. 
The corresponding \emph{Coxeter graph} $\Gamma$ of a Coxeter matrix $(m_{s,t})_{s,t \in \Gamma_0}$ is the simplicial graph with set of vertices $\Gamma_0 = S$ and set of labelled edges $\Gamma_1$, where each pair of vertices $s$ and $t$ with $m_{s,t} \geq 3$ is connected by a (unique) edge labelled by $m_{s,t}$ (including $\infty$); two distinct vertices $s$ and $t$ that are not connected by an edge has $m_{s,t} = 2$.
The \emph{Coxeter group} $\bbW \coloneqq \bbW(\Gamma)$ of a Coxeter graph $\Gamma$ is the group with the group presentation:
\[
\bbW = \< S  \mid \forall s,t \in S: s^2 = 1, (st)^{m_{s,t}} = 1\>.
\]
By convention, when $m_{s,t} = \infty$ there are no relations between $s$ and $t$.
The pair $(\bbW,S)$ is also known as a \emph{Coxeter system}.
%

\noindent
The \emph{Artin--Tits group}, or the \emph{generalised braid group}, $\bbB\coloneqq \bbB(\Gamma)$ of a Coxeter graph $\Gamma$ is the group with the group presentation:
\[
\bbB = \< \sigma_s, \  s \in S  \mid \forall s,t \in S: \underbrace{\sigma_s \sigma_t \sigma_s ...}_{m_{s,t} \text{ times}} = \underbrace{\sigma_t \sigma_s \sigma_t ...}_{m_{s,t} \text{ times}} \>.
\]
As before, there are no relations between $\sigma_s$ and $\sigma_t$ when $m_{s,t} = \infty$.
\end{definition}

Throughout this paper, all rings considered will be unital and associative; however, they need not be commutative.
Ring homomorphisms are always unital.
The following definition is adapted from the usual definition of realisations (see e.g.\ the appendix of \cite{elias_2015}), but we shall allow for non-commutative rings $R$ equipped with an anti-involution; the usual definition for commutative rings can be recovered by choosing the anti-involution to be the identity.
\begin{definition}\label{defn:realisation}
Let $R$ be a ring equipped with an anti-involution $?^*:R\to R$.
Let $(\bbW,S)$ be a Coxeter system.
Consider the free left $R$-module of rank $|S|$:
\[
R\Lambda_S := \bigoplus_{s \in S} R\cdot \alpha_s.
\]
A \emph{realisation over $R$}, or an \emph{$R$-realisation}, of $(\bbW, S)$ is a $\bbZ$-bilinear form on $R\Lambda_S$
\[
B_R(-,-): R\Lambda_S \times R\Lambda_S \ra R
\]
satisfying
\begin{enumerate}
\item $B_R(x \cdot \alpha_s, y\cdot \alpha_t) = yB_R(\alpha_s,\alpha_t)x^*$;
\item $B_R(\alpha_s, \alpha_s) = 2$ for all $s \in S$; and
\item for each $s \in S$ and $v \in R\Lambda_S$, the assignment
\begin{equation} \label{eqn:realisationrepresentation}
s(v) := v - B_R(\alpha_s,v)\cdot \alpha_s
\end{equation}
induces a well-defined left $R$-linear action of $\bbW$ on $R\Lambda_S$.
\end{enumerate}
The $|S|\times |S|$ matrix $(r_{s,t})_{s,t\in S}$ with entries $r_{s,t} \coloneqq B_R(\alpha_s, \alpha_t) \in R$ is called the \emph{Cartan matrix} and $B_R(-,-)$ is called a \emph{$R$-sesquilinear form}.
We say the realisation is \emph{faithful} if the representation defined by \eqref{eqn:realisationrepresentation} is a faithful left $R$-linear representation of $\bbW$.
\end{definition}

The following remark explains some subtleties involving realisations over non-commutative rings $R$; these subtleties can be ignored if $R$ is commutative (which is the case for the main examples we consider in \cref{sec:maineg}).
\begin{remark}[Non-commutative subtlety in matrix representations]
    If we adhere to the usual matrix multiplication rule, matrices have to act on row vectors in order for the action to be \emph{left} $R$-linear for non-commutative rings $R$.
    
    More precisely, fix an order on $S$ so that $\{\alpha_s\}_{s\in S}$ forms an ordered basis for $R\Lambda_S$ and let $A$ denote the Cartan matrix obtained with respect to this order.
    Given $v \in R\Lambda_S$, let $[v]$ denote its \emph{row vector} (not column) with entries in $R$ with respect to this ordered basis. Then we recover the sesquilinear form from $A$ via
    \begin{equation} \label{eqn:cartantobilinear}
    B_R(u,v) = [v]A[u]^\dagger,
    \end{equation}
    where $[u]^\dagger$ is the *-transpose of $[u]$, i.e.\ it is the transpose of $[u]$ with anti-involution $?^*$ applied to all entries.
    The matrix associated with the action of each $s \in S$ is given by $[s] \coloneqq \id - A[\alpha_s]^\dagger[\alpha_s]$, where $\id$ is the identity matrix -- importantly, these matrices act on row vectors, hence act on the right: $[s(v)] = [v][s]$.
    In particular, one checks that the matrix $[ts]$ representing the action of $ts$ is given by $[s][t]$, so that
    \[
    [ts(v)] = [v][ts] = [v][s][t].
    \]
\end{remark}

\subsection{Realisations over \texorpdfstring{$\bbR$}{real numbers}}
Every Coxeter system $(\bbW,S)$ carries a faithful realisation over the real numbers $\bbR$ -- with identity as involution -- given by the symmetric $\bbR$-bilinear form:
\begin{align*}
B_\bbR(-,-)&: \bbR\Lambda_S \times \bbR\Lambda_S \ra \bbR \\
B_\bbR(\alpha_s,\alpha_t) &= -2\cos(\pi/m_{s,t}),
\end{align*}
where by convention, $-2\cos(\pi/\infty) = -2$.
The corresponding representation of $\bbW$ on $\bbR\Lambda_S$ is commonly known as the \emph{Tits representation}.

In fact, we can generalise this slightly, where for $s, t \in S$ with $m_{s,t} = \infty$ we just require that $B_\bbR(\alpha_s,\alpha_t) = B_\bbR(\alpha_t,\alpha_s) \leq -2$.
This would still define a faithful realisation of $(\bbW,S)$ over $\bbR$.
We record this result for later purposes.
\begin{proposition}[see e.g.\ \protect{\cite[Theorem 4.2.7]{bjorner_anders_brenti_2010}}] \label{prop:realfaithfulrealisation}
Let $(\bbW,S)$ be a Coxeter system.
Consider the bilinear form on $\bbR\Lambda_S$ associated to a $|S|\times|S|$ symmetric real matrix $(r_{s,t})_{s,t\in S}$ satisfying the following conditions:
	\begin{itemize}
    \item for all $s\in S$, $r_{s,s} = 2$;
	\item if $2 \leq m_{s,t} < \infty$, $r_{s,t} = -2\cos(\pi/m_{s,t})$;
	\item if $m_{s,t} = \infty$, $r_{s,t} \leq -2$.
	\end{itemize}
Then the induced representation of $\bbW$ on $\bbR\Lambda_S$ and the corresponding contragredient representation of $\bbW$ on $\Hom_\bbR(\bbR\Lambda_S,\bbR)$ are both faithful.
In particular, $\bbR\Lambda_S$ faithfully realises $(\bbW,S)$ over $\bbR$.
\end{proposition}
\begin{definition} \label{defn:geometricrealrep}
A (faithful) $\bbR$-realisation $\bbR\Lambda_S$ satisfying the conditions in the proposition is called a (symmetric) \emph{geometric $\bbR$-realisation} of $\bbW$; the representation on $\bbR\Lambda_S$ is also called the (symmetric) \emph{geometric $\bbR$-representation} of $\bbW$.
\end{definition}

\begin{remark}
    If the Coxeter matrix only has off-diagonal entries $2,3,4,6$ and $\infty$, there is also a $\bbZ$-realisation whose Cartan matrix (in the sense of \cref{defn:realisation}) is given by what is usually known as a \emph{generalised Cartan matrix}, where the proposition above still holds. For example, this is the case when $\bbW$ is a crystallograhic Coxeter group or a Weyl group associated to a Kac--Moody algebra. 
\end{remark}

\section{Realisations over fusion rings}

\subsection{Fusion rings}
In addition to faithful realisations over $\bbR$, we will also study faithful realisations over a special type of rings known as \emph{fusion rings}, whose definition we now recall.
We refer the reader to \cite{EGNO15} for more details on fusion rings.
\begin{definition} \label{defn:fusionring}
A \emph{unital $\bbZ_{\geq 0}$-ring $R$ of finite rank $n$} is a free $\bbZ$-module with basis $\mathfrak{B} := \{b_j\}_{j=1}^n$ whose unital ring structure (not necessarily commutative) satisfies
\begin{enumerate}
\item $b_1 = 1$ is the unit of $R$; and
\item (positivity) $b_j\cdot b_k = \sum_{b_i \in \mathfrak{B}} {}_jc_{k,i} b_i$ is non-zero and ${}_jc_{k,i} \in \bbZ_{\geq 0}$. 
\end{enumerate}
A \emph{fusion ring} is a unital $\bbZ_{\geq 0}$-ring $R$ of finite rank equipped with an involution on the basis $\mathfrak{B}$
\[
b_j \mapsto b_{j^*} \eqqcolon b_j^* \in \mathfrak{B}
\]
such that
\begin{enumerate}
\item the induced $\bbZ$-linear map $?^*: R \ra R$ (by extending $\bbZ$-linearly) is an anti-involution; i.e.\ $(r \cdot s)^* = s^* \cdot r^*$; and
\item the unit $b_1 = 1$ appears in $b_j \cdot b_k$ with coefficient $1$ if and only if $k=j^*$; i.e.\ ${}_jc_{k,1} = 1$ if and only if $k=j^*$.
\end{enumerate}
\end{definition}
\begin{remark}
    If we drop the assumption that $1 \in \mathfrak{B}$, we obtain the notion of a (finite rank $\bbZ_{\geq 0}$-)based ring considered by Lusztig \cite{Lusztig_basedring}.
\end{remark}

The multiplicity coefficients ${}_jc_{k,i} \in \bbZ_{\geq 0}$ defined by
\begin{equation}\label{eqn:multcoefficient}
b_j \cdot b_k = \sum_i {}_jc_{k,i}b_i, \quad \forall b_j,b_k \in \mathfrak{B}
\end{equation}
satisfy the following property known as the Frobenius reciprocity:
\begin{lemma}\label{lem:Frobeniusrecip}
The coefficients in \eqref{eqn:multcoefficient} satisfy 
\[
{}_jc_{k,i} = {}_{j^*}c_{i,k} = {}_ic_{k^*,j}.
\]
\end{lemma}
\begin{proof}
See \cite[Proposition 3.1.6]{EGNO15} and use the fact that ${}_jc_{k,i} = {}_{k^*}c_{j^*,i^*}$ (apply the anti-involution to \eqref{eqn:multcoefficient}).
\end{proof}

\begin{definition}
An element $r$ in a fusion ring $R$ is said to be \emph{non-negative} if $r$ is a non-negative $\bbZ$-linear combination of the basis elements; it is \emph{positive} if $r$ is moreover non-zero.
The set of non-negative and the set of positive elements are denoted by $R_{\geq 0}$ and $R_{>0}$ respectively.
\end{definition}
By definition, both $R_{\geq 0}$ and $R_{>0}$ are closed under addition; moreover, the positivity structure and the definition of the involution ensure that they are also closed under multiplication and involution.

\begin{remark}
Note that fusion rings, being free $\bbZ$-modules, are automatically torsion free; i.e. $r + \cdots + r \neq 0$ for any number of (finite) repeated sum of $r \in R \setminus \{0\}$.
\end{remark}

\begin{example} \label{eg:groupfusionring}
	Let $G$ be a finite group (not necessarily abelian).
	The group ring $\bbZ[G]$ of $G$ with basis $\mathfrak{B} = G$ and involution given by taking the inverse element makes $\bbZ[G]$ a fusion ring.
    The ring of integers $\bbZ$ is therefore also a fusion ring (take $G$ the trivial group).
    Note that if $G$ is non-abelian, then the fusion ring $\bbZ[G]$ is non-commutative.
\end{example}

\begin{example}[Type $A_{n-1}$ Verlinde fusion ring] \label{eg:Verlinde}
The Chebyshev polynomials $\Delta_n(x) \in \bbZ[x]$ are a sequence of polynomials defined recursively by the following:
\begin{align*}
\Delta_0(x) &= 1, \qquad \Delta_1(x) = x, \\
\Delta_{k+1}(x) &= \Delta_1(x)\Delta_k(x) - \Delta_{k-1}(x).
\end{align*}
In particular, we have $\Delta_2(x) = x^2 - 1, \Delta_3(x) = x^3 - 2x, \Delta_4(x) = x^4 - 3x^2 + 1$ and so forth.
Let $R_n := \bbZ[x_n]/ \< \Delta_{n-1}(x_n) \>$ denote the quotient ring.
Using the recurrence relation above, one obtains the following multiplication rule in $R_n$ (for notational simplicity we drop ``$(x_n)$''):
\begin{equation} \label{eqn:multrule}
\Delta_j \cdot \Delta_k = 
	\begin{cases}
	\Delta_{|j-k|} + \Delta_{|j-k| + 2} + \cdots + \Delta_{j+k}, &\text{if } j+k \leq n-2; \\
	\Delta_{|j-k|} + \Delta_{|j-k| + 2} + \cdots + \Delta_{2(n-1)-(j+k)}, &\text{otherwise}.
	\end{cases}
\end{equation}
By choosing the basis $\{\Delta_{j-1}(x_n)\}_{j=1}^{n-1}$ (which contains the unit $1 = \Delta_0(x_n)$) and the involution given by the identity map; i.e.\ $\Delta_{j-1}(x_n)^* := \Delta_{j-1}(x_n)$, it follows that the ring $R_n$ is a (commutative) fusion ring of rank $n-1$.
\end{example}

\begin{example}
Let $R^{\even}_n \subset R_n$ denote the free $\bbZ$-submodule generated by the basis elements $\Delta_{2k}(x_n)$ of even degree.
The multiplication rule in \eqref{eqn:multrule} shows that $R^{\even}_n$ is in fact a subring of $R_n$.
It follows immediately that $R^{\even}_n$ is itself a fusion ring (with basis $\{\Delta_{2k}(x_n) \mid 0 \leq 2k \leq n-2\}$).
\end{example}

\begin{example}[Representation ring of $S_3$] \label{eg:repringS3}
Let $R(S_3)$ denote the representation ring of the symmetric group $S_3$ on 3 elements.
As an abelian group, $R(S_3)$ is free of rank 3, with basis elements given by the 3 isomorphism classes of irreducible representations $\mathfrak{B} := \{1,V,\sgn\}$, denoting the trivial, standard (2-dimensional) and sign representations respectively.
The commutative ring structure is determined by the tensor product of representations, which is given by
\begin{equation} \label{eqn:repS3fusionrule}
\sgn \cdot \sgn = 1, \quad \sgn\cdot V = V = V \cdot \sgn, \quad V \cdot V = 1 + V + \sgn,
\end{equation}
where $1$ is the unit.
The involution on $\mathfrak{B}$ is given by taking the dual representation, which in this case is the identity map.
It follows that $R(S_3)$ is a fusion ring with basis $\mathfrak{B}$.

More generally, the representation ring of any finite group (or any semisimple Hopf algebra) is a fusion ring, with basis the set of isomorphism classes of irreducible representations, multiplication defined by the tensor product of representations, and involution given by taking the dual representation (which need not be the identity).
\end{example}

All of the examples presented above are known as fusion rings that are \emph{categorifiable} \cite[Definition 4.10.1]{EGNO15}.
Namely, they are fusion rings that arise as Grothendieck rings of specific type of monoidal categories known as \emph{fusion categories} \cite[Definition 4.1.1]{EGNO15}.
The following provides examples of fusion rings that are not categorifiable.

\begin{example}
This is essentially an extension of \cref{eg:groupfusionring}, first studied by Tambara--Yamagami in \cite{TY_category}.
Let $G$ be a finite group and let $m$ be a formal new element.
The fusion ring $TY(G)$ is the free abelian group of rank $|G|+1$:
\[
TY(G) \coloneqq \bbZ[G] \oplus \bbZ\cdot m,
\]
with multiplication rule given by ($g,h \in G$):
\[
g\cdot h = gh; \qquad g\cdot m = m = m \cdot g; \qquad m\cdot m = \sum_{g \in G} g.
\]
The involution on $g \in G$ is again the inverse, whereas the involution of $m$ is itself.
The main theorem in \cite[Corollary 3.3]{TY_category} shows that if $G$ is a non-abelian group, then $TY(G)$ is not categorifiable.
\end{example}

\begin{example}[Tensor product of fusion rings]
Let $R$ and $R'$ be two fusion rings with basis $\mathfrak{B}$ and $\mathfrak{B}'$ respectively.
Their tensor product $R\otimes R'$ (over $\bbZ$) is again a fusion ring, with basis $\mathfrak{B} \otimes \mathfrak{B}' := \{ b\otimes b' \mid b \in \mathfrak{B}, b' \in \mathfrak{B}'\}$.
\end{example}

\subsection{Frobenius--Perron dimension}
\begin{definition}\label{defn:FPdim}
Let $R$ be a fusion ring with basis $\mathfrak{B}$.
For each $b_j \in \mathfrak{B}$, the $n\times n$ matrix $({}_jc_{k,\ell})_{1 \leq k, \ell \leq n}$ is a non-negative matrix.
By the Frobenius--Perron theorem, the matrix $({}_jc_{k,\ell})_{1 \leq k, \ell \leq n}$ has a non-negative real eigenvalue dominating the absolute value of all of its other eigenvalues.
The \emph{Frobenius--Perron dimension} $\FPdim(b_j)$ of $b_j$ is defined to be this unique eigenvalue.
\end{definition}

The following results about fusion rings will be crucial; see \cite[Proposition 3.3.6, 3.3.9 \& Corollary 3.3.16]{EGNO15}:
\begin{proposition} 
\label{prop:FPdimproperties}
Let $R$ be a fusion ring with basis $\mathfrak{B}$.
\begin{enumerate}
\item \label{item:FPdimuniqueness} The group homomorphism $\FPdim: R \ra \bbR$ defined by sending each basis element $b_j \in \mathfrak{B}$ to $\FPdim(b_j)$ is moreover a ring homomorphism. 
\item $\FPdim$ is the unique ring homomorphism satisfying the property $\FPdim(b_j) \geq 0$ for all $b_j \in \mathfrak{B}$. Moreover, $\FPdim(b_j) \geq 1$ for all $b_j \in \mathfrak{B}$.
\item $\FPdim(r) = \FPdim(r^*)$ for all $r \in R$.
\item For any basis element $b \in \mathfrak{B}$, we have $\FPdim(b) < 2 \implies \FPdim(b) = 2\cos(\pi/m)$ for some integer $m \geq 3$.
\end{enumerate}
\end{proposition}

\begin{remark}
    Note that the proposition above implies that any positive element $r \in R_{> 0}$ with $\FPdim(r) < 2$ must be equal to some basis element $b \in \mathfrak{B}$.
\end{remark}

\begin{example}
For $R = \bbZ[G]$, each basis element $g \in G$ has $\FPdim(g) = 1$.
For $R = TY(G)$, we have, moreover, that $\FPdim(m) = \sqrt{|G|}$.
\end{example}

\begin{example}
For each integer $n \geq 2$, the roots of the Chebyshev polynomial $\Delta_{n-1}(x)$ are given by $2\cos(k\pi/n)$ for integers $1 \leq k \leq n-1$.
In particular, the assignment $\Delta_1(x_n) = x_n \mapsto 2\cos(\pi/n)$ is a well-defined ring homomorphism from $R_n$ to $\bbR$.
Moreover, this satisfies the property that $\Delta_j(2\cos(\pi/n)) \geq 0$ for all $0 \leq j \leq n-1$.
To see this, recall that the $k$-th quantum number is defined as
\[
[k]_q = \frac{q^k -q^{-k}}{q-q^{-1}} = q^{k-1} + q^{k-3} + \cdots + q^{-(k-3)} + q^{-(k-1)},
\]
which satisfies the same recurrence relation: $[k+1]_q = [2]_q[k]_q+[k-1]_q$.
Then for all $1 \leq k \leq n$,
\[
\Delta_{k-1}(2\cos(\pi/n)) = [k]_{e^{i\pi/n}} \geq 0.
\]
By the uniqueness result of \cref{prop:FPdimproperties}\ref{item:FPdimuniqueness}, this agrees with the ring homomorphism $\FPdim: R_n \ra \bbR$, i.e.\ $\FPdim(\Delta_j(x_n)) = \Delta_j(2\cos(\pi/n))$ for all basis elements $\Delta_j(x_n) \in R_n$.
\end{example}

\begin{example}
Since $R^{\even}_n$ is a subring of $R_n$ freely generated by a subset of the basis of $R_n$, it follows (again by uniqueness) that $\FPdim: R^{\even}_n \ra \bbR$ is also defined by $\FPdim(\Delta_{2k}(x_n)) = \Delta_{2k}(2\cos(\pi/n))$.
Suppose that $n \geq 2$.
Using \eqref{eqn:multrule}, we have that 
\[
\begin{cases}
\Delta_{n-2}(x_n) \cdot \Delta_{n-2}(x_n) = \Delta_0(x_n) = 1 \\
\Delta_{n-2}(x_n) \cdot \Delta_{n-3}(x_n) = \Delta_1(x_n) = x_n.
\end{cases}
\]
These imply that $\FPdim(\Delta_{n-2}(x_n)) = 1$ and $\FPdim(\Delta_{n-3}(x_n)) = 2\cos(\pi/n)$.
Note that in particular, $R^{\even}_3 = \bbZ\cdot\{\Delta_0(x_3)\} \cong \bbZ$.
\end{example}

\begin{example}
For the representation ring of any finite group (or any semisimple Hopf algebra), the underlying dimension of the representation defines a ring homomorphism that sends each irreducible representation to a positive integer, hence agrees with $\FPdim$ by uniqueness.
In particular, for $R(S_3)$ the representation ring of $S_3$, we have
\[
\FPdim(1) = 1, \quad \FPdim(V) = 2, \quad \FPdim(\sgn) = 1.
\]
\end{example}

\begin{example}
Given a tensor product of two fusion rings $R \otimes R'$, we have that $\FPdim(r \otimes r') = \FPdim(r) \FPdim(r')$.
\end{example}

\subsection{Geometric realisations over fusion rings}
Let $(\bbW,S)$ be a Coxeter system.
We will be mainly interested in realisations over fusion rings satisfying the following conditions.
\begin{definition} \label{defn:geometricfusionrep}
    Let $R\Lambda_S$ be a realisation of $(\bbW,S)$ over the fusion ring $R$, with Cartan matrix $(r_{s,t})_{s,t \in S}$.
    We say that the $R$-realisation $R\Lambda_S$ is \emph{geometric} if
    \begin{enumerate}
		\item for all $s \in S$, $r_{s,s} = 2$ (this is implied by the definition of an $R$-realisation);
    	\item \label{item:*skewsym}for all $s \neq t$, $-r_{s,t} \in R_{\geq 0}$ and $r_{t,s} = r_{s,t}^*$;  and
    	\item each $r_{s,t}$ satisfies    
    		\begin{equation} \label{eqn:FPdimcondition}
    		\FPdim(r_{s,t}) = 
	   			\begin{cases}
			   	2, &\text{ if } s=t; \\
	   			-2\cos(\pi/m_{s,t}), &\text{ if } m_{s,t}<\infty; \\
	   			r_{s,t} \leq -2, &\text{ if } m_{s,t} = \infty.
	   			\end{cases}
    		\end{equation}
	\end{enumerate} 
    The corresponding $R$-representation $R\Lambda_S$ is also called a \emph{geometric $R$-representation} of $\bbW$.
\end{definition}
One readily checks that property \ref{item:*skewsym} of geometric $R$-realisations ensures that the $\bbW$-action on $R\Lambda_S$ preserves the $R$-sesquilinear form $B_R(-,-)$; i.e.
\[
B_R(s(u),s(v)) = B_R(u,v)
\]
for all $s \in S \subseteq \bbW$.
Note that while bilinear forms of geometric $\bbR$-realisations are, by definition, always symmetric, sesquilinear forms of geometric $R$-realisations are instead *-symmetric with respect to the anti-involution on $R$:
\[
B_R(u,v) = B_R(v,u)^*
\]

Moreover, we see that $(\FPdim(r_{s,t}))_{s,t \in S}$ is the Cartan matrix of a faithful geometric $\bbR$-realisation of $(\bbW,S)$, as in \cref{prop:realfaithfulrealisation}.
We will explicitly relate geometric $R$-realisations and geometric $\bbR$-realisations in later sections.

Examples of geometric realisations of Coxeter systems over \emph{categorifiable fusion rings} can be obtained via the following result (cf.\ \cref{prop:realfaithfulrealisation}) deducible from the technology of fusion quivers introduced in \cite{EH_fusionquivers}.
Since we will provide separate proofs for the main cases of interest in \cref{sec:maineg} (and we expect the assumption of the fusion rings being categorifiable to be superfluous; see \cref{conj:fusioncatsuperfluous}), in the interest of space we shall only sketch the main ideas and omit definitions from \cite{EH_fusionquivers}.
\begin{theorem} \label{thm:fusionrealisationcondition}
    Suppose $R$ is a categorifiable fusion ring. Let $(r_{s,t})_{s,t \in S}$ be a matrix with entries in $R$ satisfying the geometric conditions above; i.e
        \begin{enumerate}
		\item for all $s \in S$, $r_{s,s} = 2$;
    	\item for all $s \neq t$, $-r_{s,t} \in R_{\geq 0}$ and $r_{t,s} = r_{s,t}^*$;  and
    	\item each $r_{s,t}$ satisfies
            \[
    		\FPdim(r_{s,t}) = 
	   			\begin{cases}
			   	2, &\text{ if } s=t; \\
	   			-2\cos(\pi/m_{s,t}), &\text{ if } m_{s,t}<\infty; \\
	   			r_{s,t} \leq -2, &\text{ if } m_{s,t} = \infty.
	   			\end{cases}
            \]
	   \end{enumerate} 
    Then the $R$-sesquilinear form on $R\Lambda_S$ defined by this matrix according to \eqref{eqn:cartantobilinear} is a $R$-realisation of $(\bbW,S)$.
    This $R$-realisation -- geometric by definition -- is moreover faithful.
\end{theorem}
\begin{proof}[Proof sketch]
    Suppose $R$ is a categorifiable fusion ring, say $R = [\mathcal{C}]$ is the fusion ring of a fusion category $\mathcal{C}$.
    Then $R\Lambda_S$ can be realised as the Grothendieck group $[\mathrm{Rep}_{\mathcal{C}}(Q)]$ of the category of representations of some fusion quiver $Q$ over $\mathcal{C}$, whose associated Coxeter graph $\Gamma_Q$ is the Coxeter graph of $(\bbW,S)$.
    By \cite[Theorem 6.8]{EH_fusionquivers}, $\bbW = \bbW(\Gamma_Q)$ is isomorphic to the group $\bbW(Q)$ (see \cite[Definition 4.24]{EH_fusionquivers}), which is the same as saying $\bbW$ acts faithfully on $R\Lambda_S = [\mathrm{Rep}_{\mathcal{C}}(Q)]$ via \eqref{eqn:realisationrepresentation} in \cref{defn:realisation}.
\end{proof}
Note the extra assumption that $R$ is a categorifiable fusion ring -- we expect this assumption to be superfluous:
\begin{conjecture}
    \label{conj:fusioncatsuperfluous}
    \cref{thm:fusionrealisationcondition} holds for any fusion ring $R$.
\end{conjecture}
We will later prove the faithfulness part of the conjecture above as \cref{cor:faithfulfusionrealisation}; i.e.\ geometric $R$-realisations over any fusion ring $R$ are faithful.
\begin{remark}
The only part of the proof of \cite[Theorem 6.8]{EH_fusionquivers} (used in \cref{thm:fusionrealisationcondition}) that requires the assumption $R$ to be categorifiable is a positivity lemma in relation to (not necessarily commutative) two variable quantum numbers; see \S 5 and Lemma 5.7 in \cite{EH_fusionquivers} for more details.
\end{remark}

\subsection{Main examples}\label{sec:maineg}
We provide here some examples of geometric $R$-realisations obtainable from \cref{thm:fusionrealisationcondition}.
For completeness and for simplicity, we will sketch the proofs that these are indeed geometric $R$-realisations.
The following example shows that every Coxeter system has a geometric realisation over some fusion ring (note that $\bbR$ is not a fusion ring, so this does not follow from classical theory).
\begin{example} \label{eg:Verlinderealisation}
Let $(\bbW,S)$ be a Coxeter system and consider the (finite) set
\[
M:= \{ m_{s,t} \mid s \neq t \in S\} \subset \{2,3,4,\ldots\} \cup \{\infty\}.
\]
Let $\Delta_i(x)$ be the $i$-th Chebyshev polynomial; let us also set $\Delta_{\infty-1}(x_\infty) \coloneqq x_\infty - 2$ by convention.
For each $m \in M$, consider the fusion ring
\[
R_m \coloneqq \bbZ[x_m]/\< \Delta_{m-1}(x_m) \>,
\]
which is the Verlinde fusion ring from \cref{eg:Verlinde} for all $m < \infty$, and $R_\infty \coloneq \bbZ[x_\infty]/\<x_\infty-2\> \cong \bbZ$.
We define the type $A$ fusion ring $R_M$ associated to $(\bbW,S)$ as
\[
R_M := \bigotimes_{m \in M} R_m \cong 
	\bbZ[x_m \mid m \in M]/\< \Delta_{m-1}(x_m) \text{ for each } m \in M\>.
\]
(In the degenerate case where $M = \emptyset$, we set $R_M := \bbZ$).
By construction, $R_M$ is a commutative fusion ring whose involution is the identity map.

For notational simplicity, we will use the identification of $R_M$ as the quotient ring of the polynomial ring given above.
As such, every $r \in R_m$ for $m \in M$ is immediately an element of $R_M$.
Moreover, each basis element of $R_M$ is the product of basis elements $\Delta_{j_m}(x_m)$ of each $R_m$ (once again, if $M = \emptyset$ then the basis element of $R_M$ is just $1$).

Consider the (Cartan) matrix $(r_{s,t})_{s,t\in S}$ defined by
\[
r_{s,t} = r_{t,s} \coloneqq
	\begin{cases}
	2, &\text{if } s=t; \\
	-x_{m_{s,t}}, &\text{if } s \neq t;
	\end{cases}
\]
note that $x_2 = \Delta_1(x_2) = 0$ and $x_\infty = 2$ by definition.
The corresponding $R_M$-sesquilinear form on $R_M\Lambda_S$ is therefore defined on the basis elements by
\[
B_{R_M}(\alpha_s,\alpha_t) = r_{s,t}.
\]
Note that this $R_M$-sesquilinear form is in fact a symmetric $R_M$-bilinear form, since the involution $?^*$ is the identity.

\begin{proposition}
    $R_M\Lambda_S$ equipped with $B_{R_M}(-,-)$ is a geometric $R_M$-realisation of $(\bbW,S)$.
\end{proposition}
\begin{proof}
    The proof is a direct application of \cite[A.2]{elias_2015}.
    Namely, we only have to show that for each pair of distinct generators $s \neq t \in S$ with $m \coloneqq m_{s,t} < \infty$, $(st)^m$ acts on $R_M\Lambda_S$ by the identity (it is clear that $s^2 = \id$, and if $m_{s,t} = \infty$ there is nothing to check).
    To this end, note that for all $k \geq 1$, we have
    \begin{equation} \label{eqn:stformula}
    \begin{split}
    (st)^k(\alpha_s) &= \Delta_{2k}(x_m)\alpha_s + \Delta_{2k-1}(x_m)\alpha_t; \\
    (st)^k(\alpha_t) &= -\Delta_{2k-1}(x_m)\alpha_s - \Delta_{2k-2}(x_m)\alpha_t; \\
    (st)^k(\alpha_u) &= \alpha_u + (\Delta_{k-1}(x_m)^2r_{s,u} + \Delta_{k-1}(x_m)\Delta_{k}(x_m)r_{t,u})\alpha_s \\
    & \quad - (\Delta_{k-1}(x_m)\Delta_{k-2}(x_m)r_{s,u}+\Delta_{k-1}(x_m)^2r_{t,u})\alpha_t;
    \end{split}
    \end{equation}
    where $u \notin \{s,t\}$.
    In particular, if $k=m$, we see that $(st)^m(\alpha_v) = \alpha_v$ for all $v \in S$.
\end{proof}
\end{example}

\begin{example} \label{eg:modifiedrealisation}
The following modifications on $R_M$ also provide geometric realisations of $(\bbW,S)$ over fusion rings (which are also commutative with $?^*$ the identity):
\begin{itemize}
\item Let $R^e_M$ denote the fusion ring obtained from $R_M$ by replacing $R_m$ with $R_m^{\even}$ for each odd $m$ in the definition of $R_M$ above. Moreover, for each pair of distinct $s$ and $t$ with $m_{s,t}$ odd, we set 
\[
r_{s,t} \coloneqq -\Delta_{m-3}(x_m) 
\]
in place of $-x_m = -\Delta_1(x_m)$.
\item Let $R^\infty_M$ be the fusion ring obtained from $R_M$ by replacing $R_\infty$ by the representation ring $R(S_3)$ of $S_3$ (see \cref{eg:repringS3}).
Moreover, for each pair of distinct $s$ and $t$ with $m_{s,t}=\infty$, we set $r_{s,t} \coloneqq -V$ in place of $-x_\infty = -2$, where $V$ is the two-dimensional standard representation of $S_3$.
\end{itemize}
The fact that $R^e_M\Lambda_S$ is still a geometric realisation of $\bbW$ follows from the same type of computation in \eqref{eqn:stformula}.
On the other hand there is nothing left to check for $R^\infty_M\Lambda_S$ since we only modified the $\infty$ part.
\end{example}

\begin{remark}
    We emphasise that in \cref{eg:modifiedrealisation} we only replace $R_m$ by $R^{\even}_m$ for odd $m$, since there may not be faithful geometric realisations (in the sense of \cref{defn:geometricfusionrep}) over $R^{\even}_m$ for even $m$.
    For example, for $M = \{4\}$, we need an element with $\FPdim$ equals to $2\cos(\pi/4)=\sqrt{2}$, but $R^{\even}_4$ has basis $\{\Delta_0=1, \Delta_2\}$ which both have $\FPdim$ equal to $1$.
\end{remark}

\begin{example} \label{eg:TYrealisation}
For labels $m_{s,t}=4$ and $6$, one could also use $TY\left(\bbZ/\left(\frac{m_{s,t}}{2}\bbZ\right)\right)$ in place of $R_{m_{s,t}}$, and use $m \in TY\left(\bbZ/\left(\frac{m_{s,t}}{2}\bbZ\right)\right)$ in place of $x_{m_{s,t}}$; note that $\FPdim(m) \geq 2$ in $TY(\bbZ/a\bbZ)$ for $a \geq 4$, which will only result in realisations for $m_{s,t}=\infty$.
We leave it to the reader to check that these also provide geometric realisations over fusion rings.
\end{example}

\subsection{From geometric realisations over fusion rings to realisations over integers} \label{sec:fusionreptoZrep}
Throughout this section, let $R$ be a fusion ring with basis $\mathfrak{B}$ and let $R\Lambda_S$ be a geometric $R$-realisation of $(\bbW,S)$, with $(r_{s,t})_{s,t \in S}$ its Cartan matrix. 

Recall that $R$, being a fusion ring, is by definition a free $\bbZ$-module with basis $\mathfrak{B}$.
It follows that $R\Lambda_S$ is a free $\bbZ$-module
with the following canonical decomposition: 
\begin{equation} \label{eqn:Zmodstructure}
R\Lambda_S = \bigoplus_{b \in \mathfrak{B}, s \in S} \bbZ\cdot b \alpha_s,
\end{equation}
where $\{b\alpha_s \mid b \in \mathfrak{B}, s \in S\}$ is the $\bbZ$-basis.

The trivial yet important observation is that any $R$-linear morphism is automatically $\bbZ$-linear, i.e.\ a morphism of abelian groups.
As such, given any realisation over the fusion ring $R$, the corresponding $R$-linear action of $\bbW$ on $R\Lambda_S$ is also $\bbZ$-linear.

We shall describe the $\bbZ$-linear action of $\bbW$ on $R\Lambda_S$ in terms of a ``larger'' Coxeter system realisable by an integral Cartan matrix.
We begin with the following abstract definition:
\begin{definition}
\label{defn:unfoldedCox}
Let $(\bbW, S)$ be a Coxeter system together with a geometric $R$-realisation $R\Lambda_S$, and let $(r_{s,t})_{s,t \in S}$ denote its Cartan matrix.
The \emph{unfolded Coxeter system} $(\check{\bbW}, \check{S})$ with respect to the realisation $R\Lambda_S$ is the Coxeter system corresponding to the Coxeter graph $\check{\Gamma}$ defined as follows:
\begin{itemize}
\item the set of vertices $\check{\Gamma}_0$ is given by $\check{S} \coloneqq \mathfrak{B} \times S$;
\item two distinct vertices $(b_i, s), (b_j,t) \in \check{\Gamma}_0$ are connected by an edge if and only if $s\neq t$ and $b_j$ appears in $(-r_{s,t}) \cdot b_i$ as a non-zero summand; moreover this edge is labeled by $3$ if $b_j$ appears with coefficient 1, otherwise it is labeled by $\infty$.
\end{itemize}
\end{definition}
We note that the unfolded Coxeter system does depend on the realisation; see \cref{rmk:distinct_realisation}.

\begin{example} \label{eg:unfoldedsystems}
Let $R_M$ be the type $A$ fusion ring associated to $(\bbW,S)$ and consider its geometric $R_M$-realisation as in \cref{eg:Verlinderealisation}.
Let us first focus on the case where $M = \{m\}$ is a singleton set, so that $R_M = R_m$.
First consider the case $m < \infty$, so that the basis of $R_m$ is given by $\mathfrak{B} = \{\Delta_0(x_m)=1,\Delta_1(x_m)=x_m,\ldots, \Delta_{m-2}(x_m)\}.$
Recall that (see e.g.\ \eqref{eqn:multrule})
\[
x_m \cdot \Delta_i(x_m) = \Delta_1(x_m) \cdot \Delta_i(x_m) = \begin{cases}
    \Delta_1(x_m), &\text{ if } i=0 \\
    \Delta_{i-1}(x_m) + \Delta_{i+1}(x_m), &\text{ if } 1 \leq i \leq m-3; \\
    \Delta_{i-3}(x_m), &\text{ if } i = m-2.
\end{cases}
\]
From this we see that for any two distinct $s \neq t \in S$, we get two copies of type $A_{m-1}$ chains in the Coxeter graph $\check{\Gamma}$:
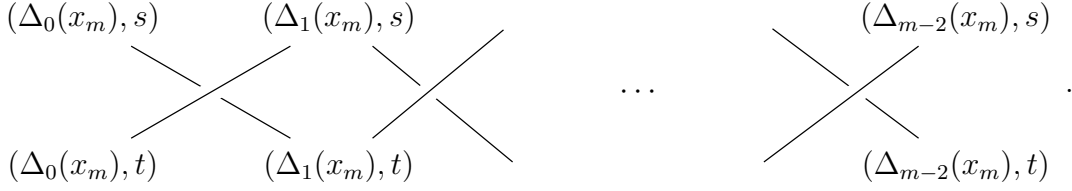
\[\begin{tikzcd}[row sep =small]
	{(\Delta_0(x_m),s)} & {(\Delta_1(x_m),s)} & {{}} && {{}} & {(\Delta_{m-2}(x_m),s)} \\
	&&& \cdots \\
	{(\Delta_0(x_m),t)} & {(\Delta_1(x_m),t)} & {{}} && {{}} & {(\Delta_{m-2}(x_m),t)}
	\arrow[no head, from=1-1, to=3-2]
	\arrow[no head, from=1-2, to=3-3]
	\arrow[no head, from=1-5, to=3-6]
	\arrow[no head, from=3-1, to=1-2]
	\arrow[no head, from=3-2, to=1-3]
	\arrow[no head, from=3-5, to=1-6]
\end{tikzcd}.\]
The case where $m = \infty$ is trivial, since $R_M = R_\infty \cong \bbZ$ and the Coxeter graph $\check{\Gamma}$ is the same as the Coxeter graph $\Gamma$.

Now consider the general case where $M$ need not be a singleton set. 
The basis elements of $R_M$ is given by a product $\prod_{m \in M} \Delta_{i_m}(x_m)$ for each possible sequence $(i_m)_{m \in M}$ with $0 \leq i_m \leq m-2$ if $m < \infty$, and $i_\infty = 0$.
Let $s \neq t \in S$ be such that $m = m_{s,t}$.
Since each $x_m$ only interacts with the $\Delta_{i_m}(x_m)$ part, for each $\prod_{\substack{m' \in M \\ m' \neq m}} \Delta_{i_{m'}}(x_{m'})$, we obtain two copies of type $A_{m-1}$ chains as above when $m < \infty$:
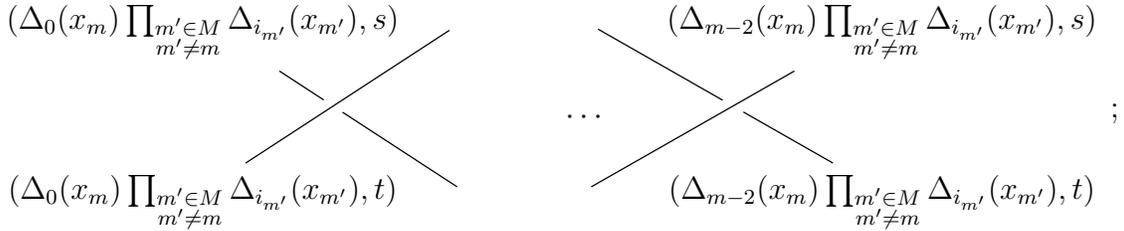
\[\begin{tikzcd}[row sep =small, column sep = small]
	{(\Delta_0(x_m)\prod_{\substack{m' \in M \\ m' \neq m}} \Delta_{i_{m'}}(x_{m'}),s)} &  {{}} && {{}} & {(\Delta_{m-2}(x_m)\prod_{\substack{m' \in M \\ m' \neq m}} \Delta_{i_{m'}}(x_{m'}),s)} \\
	&&& \cdots \\
	{(\Delta_0(x_m)\prod_{\substack{m' \in M \\ m' \neq m}} \Delta_{i_{m'}}(x_{m'}),t)} & {{}} && {{}} & {(\Delta_{m-2}(x_m)\prod_{\substack{m' \in M \\ m' \neq m}} \Delta_{i_{m'}}(x_{m'}),t)}
	\arrow[no head, from=1-1, to=3-2]
	\arrow[no head, from=1-4, to=3-5]
	\arrow[no head, from=3-1, to=1-2]
	\arrow[no head, from=3-4, to=1-5]
\end{tikzcd};\]
otherwise when $m = \infty$, we just get a copy of affine $\widetilde{A}_1$.
Note that if we use the geometric realisation over $R^e_M$ from \cref{eg:modifiedrealisation}, we get only one copy of type $A_{m-1}$ for odd $m$ (nothing changes for even $m$).
On the other hand, if we use the geometrisation over $R^\infty_M$ from \cref{eg:modifiedrealisation}, the affine $\widetilde{A}_1$ graph for $m=\infty$ is replaced with an affine $\widetilde{D}_5$ Coxeter graph 
(cf.\ \eqref{eqn:repS3fusionrule}):
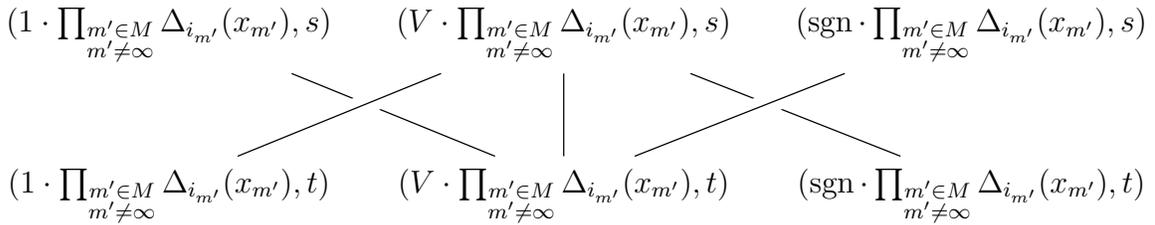
\[\begin{tikzcd}[column sep=small, row sep = large]
	(1\cdot\prod_{\substack{m' \in M \\ m' \neq \infty}} \Delta_{i_{m'}}(x_{m'}),s) & (V\cdot\prod_{\substack{m' \in M \\ m' \neq \infty}} \Delta_{i_{m'}}(x_{m'}),s) &  (\sgn\cdot\prod_{\substack{m' \in M \\ m' \neq \infty}} \Delta_{i_{m'}}(x_{m'}),s) \\
	(1\cdot\prod_{\substack{m' \in M \\ m' \neq \infty}} \Delta_{i_{m'}}(x_{m'}),t) &  (V\cdot\prod_{\substack{m' \in M \\ m' \neq \infty}} \Delta_{i_{m'}}(x_{m'}),t) &  (\sgn\cdot\prod_{\substack{m' \in M \\ m' \neq \infty}} \Delta_{i_{m'}}(x_{m'}),t)
	\arrow[no head, from=1-1, to=2-2]
	\arrow[no head, from=1-2, to=2-2]
	\arrow[no head, from=1-2, to=2-3]
	\arrow[no head, from=2-1, to=1-2]
	\arrow[no head, from=2-2, to=1-3]
\end{tikzcd}.\]
In particular, $\check{\Gamma}$ with respect to $R^\infty_M$ is always a simply-laced Coxeter graph.
\end{example}

\begin{example} \label{eg:pentagonunfolding}
    We record also the following specific case, as it will be the smallest (in rank) non-trivial example that we will refer to later.
    Let $\Gamma = I_2(5) = \begin{tikzcd}[column sep=small] s \ar[r, no head, "5"] & t \end{tikzcd}$, so that $\bbW$ is the group of isometries of a regular pentagon.
    Then $R^e_M \cong \bbZ[x]/\langle x^2 = 1+x\rangle$ (where the isomorphism identifies $\Delta_2(x_5) \mapsto x$) and $\check{\Gamma}$ with respect to the geometric $R^e_M$-realisation is given by
    \[\begin{tikzcd}
	{(1,s)} & {(x,s)} \\
	{(1,t)} & {(x,t)}
	\arrow[no head, from=1-1, to=2-2]
	\arrow[no head, from=1-2, to=2-1]
	\arrow[no head, from=2-2, to=1-2]
    \end{tikzcd},\]
    which is of type $A_4$.
\end{example}

Consider the integral matrix $(\check{r}_{(b_i,s), (b_j,t)})_{(b_i,s), (b_j,t) \in \check{S}}$ with entries defined as follows.
\begin{itemize}
\item For $s=t$,
    \[
    \check{r}_{(b_i,s),(b_j,s)} :=
        \begin{cases}
            2, &\text{if } i=j; \\
            0, &\text{if } i\neq j.
        \end{cases}
    \]
\item For $s \neq t$, $\check{r}_{(b_i,s), (b_j,t)}$ is defined by 
    \[
    b_j\cdot r_{s,t} = \sum_{b_i \in \mathfrak{B}} \check{r}_{(b_i,s), (b_j,t)} b_i.
    \]
    (This is well-defined since $\mathfrak{B}$ is a $\bbZ$-basis for $R$.)
\end{itemize}
Our assumption that the $R$-realisation $R\Lambda_S$ is geometric says that $-r_{s,t} \in R_{\geq 0}$ when $s\neq t$, hence $\check{r}_{(b_i,s), (b_j,t)} \in \bbZ_{\leq 0}$.
\begin{lemma} \label{lem:Zfaithful}
The $\bbZ$-bilinear form defined by the integral matrix 
\begin{equation}\label{eqn:unfoldCartan}
(\check{r}_{(b_i,s), (b_j,t)})_{(b_i,s), (b_j,t) \in \check{S}}
\end{equation}
is symmetric, and moreover it faithfully realises $(\check{\bbW}, \check{S})$ over $\bbZ$.
\end{lemma}
\begin{proof}
Using \cref{lem:Frobeniusrecip} combined with $r_{t,s} = r_{s,t}^*$, we get $\check{r}_{(b_i,s), (b_j,t)} = \check{r}_{(b_j,s), (b_i,t)}$, so the integral matrix is symmetric.
The lemma now follows from the observation that the integral matrix $(\check{r}_{(b_i,s), (b_j,t)})_{(b_i,s), (b_j,t) \in \check{S}}$ satisfies the conditions stated in \cref{prop:realfaithfulrealisation} -- since the bilinear form is integral, the restriction to $\bbZ\Lambda_{\check{S}} \subset \bbR\Lambda_{\check{S}}$ does not matter. 
\end{proof}
\begin{definition}
The $\bbZ$-realisation of $(\check{\bbW}, \check{S})$ with Cartan matrix in \eqref{eqn:unfoldCartan} is called the \emph{unfolded $\bbZ$-realisation} of $(\check{\bbW}, \check{S})$ with respect to $R\Lambda_S$.
The $\bbZ$-bilinear form is denoted by $B_\bbZ(-,-)$ and the realisation is denoted by $\bbZ\Lambda_{\check{S}}$.
\end{definition}
\begin{remark}
If one views $\bbZ$ as a fusion ring, the $\bbZ$-realisation above is also geometric.
\end{remark}

\begin{proposition} \label{prop:unfoldedidentification}
Let $R\Lambda_S$ be a geometric $R$-realisation of $(\bbW,S)$ and let $\bbZ\Lambda_{\check{S}}$ be the unfolded $\bbZ$-realisation of the unfolded Coxeter system $(\check{\bbW}, \check{S})$ over $\bbZ$.
The map defined on the generators by
\begin{align*}
    \varphi: \bbW &\rightarrow \check{\bbW} \\
    s &\mapsto \prod_{b_i \in \mathfrak{B}} (b_i,s)
\end{align*}
is a well-defined group homomorphism.
Moreover, the $\bbZ$-linear isomorphism defined on the $\bbZ$-basis elements via:
\begin{align*}
    \Psi:R\Lambda_S &\xrightarrow{\cong} \bbZ\Lambda_{\check{S}}:= \bigoplus_{(b_i,s)\in \check{S}} \bbZ \cdot \alpha_{b_i,s} \\
    b_i \alpha_s &\mapsto \alpha_{b_i,s}
\end{align*}
is $\bbW$-equivariant with respect to $\varphi$.
\end{proposition}
\begin{proof}
    It is clear that $\Psi$ is in fact a $\bbZ$-linear isomorphism.
    We therefore obtain a $\bbZ$-linear action of $\bbW$ on $\bbZ\Lambda_{\check{S}}$ induced by $\Psi$; namely each standard generator $s \in S \subset \bbW$ acts on $v \in \bbZ\Lambda_{\check{S}}$ by
    \[
    v \mapsto \Psi(s\cdot\Psi^{-1}(v)).
    \]
    
    We claim that for each $s \in S$, the induced action above agrees with the action of $\prod_{b_i \in \mathfrak{B}} (b_i,s) \in \check{\bbW}$.
    Since the action of $\check{\bbW}$ on $\bbZ\Lambda_{\check{S}}$ is faithful by \cref{lem:Zfaithful}, it then follows that $\varphi$ is well-defined, and that $\Psi$ is indeed $\bbW$-equivariant with respect to $\varphi$.

    Let us now prove the claim.
    Firstly, note that given $r \in R$ with
    \[
    r = \sum_{b_i \in \mathfrak{B}}c_{r,b_i} b_i, \qquad c_{r,b_i} \in \bbZ,
    \]
    the $\bbZ$-linear isomorphism $\Psi:R\Lambda_S \xrightarrow{\cong} \bbZ\Lambda_{\check{S}}$ just sends
    \[
    r \alpha_s \mapsto \sum_{b_i \in \mathfrak{B}}c_{r,b_i} \alpha_{b_i,s}
    \]
    In particular, for all $s,t \in S$ and all $b_j \in \mathfrak{B}$, it follows from the definition of $\check{r}_{(b_i,s), (b_j,t)}$ that $\Psi$ sends
    \begin{equation} \label{eqn:unfoldcoeff}
    (b_j\cdot r_{s,t})\alpha_s \mapsto \sum_{b_i \in \mathfrak{B}} \check{r}_{(b_i,s), (b_j,t)}\alpha_{b_i,s}
    \end{equation}
    Now $B_\bbZ(\alpha_{b_i,s},\alpha_{b_j,s})=0$ for all $i \neq j$ (the element $\prod_{b_i \in \mathfrak{B}} (b_i,s)$ is a product of pairwise commutative elements in $\bbW$), so we have
    \[
    \left(\prod_{b_i \in \mathfrak{B}} (b_i,s) \right) \cdot v
        = v - \sum_{b_i\in\mathfrak{B}}B_\bbZ(\alpha_{b_i,s},v)\alpha_{b_i,s}.
    \]
    For $v = \alpha_{b_j,t}$, we get
    \[
    \left(\prod_{b_i \in \mathfrak{B}} (b_i,s) \right) \cdot \alpha_{b_j,t}
        = \alpha_{b_j,t} - \sum_{b_i\in\mathfrak{B}}\check{r}_{(b_i,s), (b_j,t)}\alpha_{b_i,s}
    \]
    On the other hand, the action of $s \in S$ on $b_j\cdot\alpha_t \in R\Lambda_S$ is given by
    \[
    s \cdot (b_j\alpha_t) = 
        b_j\alpha_t - B_R(\alpha_s,b_j\alpha_t)\alpha_s =
        b_j\alpha_t - (b_j\cdot r_{s,t})\alpha_s.
    \]
    Using \eqref{eqn:unfoldcoeff}, we get that for each $s \in S \subset \bbW$ and each $v \in \bbZ\Lambda_{\check{S}}$,
    \[
    \Psi\left(s\cdot \Psi^{-1}(v)\right) = \left(\prod_{b_i\in\mathfrak{B}} (b_i,s)\right)\cdot v,
    \]
    which proves our claim.
\end{proof}

\begin{definition}\label{defn:unfoldinghomomorphism}
    The group homomorphism 
        \begin{align*}
        \varphi: \bbW &\rightarrow \check{\bbW} \\
        s &\mapsto \prod_{b_i \in \mathfrak{B}} (b_i,s)
        \end{align*}
    in \cref{prop:unfoldedidentification} is called the \emph{unfolding homomorphism} with respect to $R\Lambda_S$.
\end{definition}
We will show later in \cref{cor:Coxgroupinjection} that the unfolding homomorphism is always injective.

\section{Vinberg systems and embeddings of hyperplane complements}
Throughout this section, $(\bbW,S)$ is a Coxeter system and we fix $R\Lambda_S$ to be a geometric realisation of $(\bbW,S)$ over a fusion ring $R$ with basis $\mathfrak{B}$, with Cartan matrix $(r_{s,t})_{s,t\in S}$.
\subsection{Contragredient representations from fusion rings} \label{sec:contragredientfusionrep}
We endow $\bbR$ with the structure of a left $R$-module via the ring homomorphism
\[
\FPdim:R \ra \bbR
\]
given by the Frobenius--Perron dimension map (see \cref{defn:FPdim} and \cref{prop:FPdimproperties}).
Using this, the space of $R$-linear morphisms from $R\Lambda_S$ to $\bbR$ is well-defined, which we shall denote by $\Hom_R(R\Lambda_S, \bbR)$.
We can therefore consider the $\bbR$-linear representation of $\bbW$ on $\Hom_R(R\Lambda_S, \bbR)$, namely
\[
(w\cdot Z)(v) = Z(w^{-1}\cdot v), \qquad \text{for each } w\in \bbW,
\]
where $w^{-1}\cdot v$ is now defined by the action of $\bbW$ on $R\Lambda_S$.

Let $\bbR\Lambda_S$ be the geometric $\bbR$-realisation with Cartan matrix $(\FPdim(r_{s,t}))_{s,t\in S}$ -- recall that it is geometric since $R\Lambda_S$ is geometric.
Let $\Hom_\bbR(\bbR\Lambda_S,\bbR)$ be the corresponding (dual) contragredient representation; the definition of the $\bbW$-action on $\Hom_\bbR(\bbR\Lambda_S,\bbR)$ is similar as before.
We shall relate the representations $\Hom_R(R\Lambda_S,\bbR)$ and $\Hom_\bbR(\bbR\Lambda_S,\bbR)$ via the following map:
\begin{definition}
The group homomorphism $\widetilde{\FPdim}: R\Lambda_S \ra \bbR\Lambda_S$ is defined by sending $r\cdot \alpha_s \mapsto \FPdim(r)\alpha_s$ for each $s \in S$, and extend $\bbZ$-linearly.
\end{definition}
Note that $\widetilde{\FPdim}$ relates the two bilinear forms $B_\bbR(-,-)$ and $B_R(-,-)$ as follows:
\begin{equation} \label{eqn:bilinearformsFPdim}
\FPdim(B_R(u,v)) = B_\bbR(\widetilde{\FPdim}(u),\widetilde{\FPdim}(v))    
\end{equation}
for all $u,v \in R\Lambda_S$.
Indeed, the following are also immediate since $B_R$ and $B_\bbR$ are both bilinear forms:
\begin{align*}
    \FPdim(B_R(\alpha_s + \alpha_t, \alpha_u)) &= B_\bbR(\widetilde{\FPdim}(\alpha_s+\alpha_t),\widetilde{\FPdim}(\alpha_u)) \\
    \FPdim(B_R(\alpha_s,\alpha_t+\alpha_u)) &= B_\bbR(\widetilde{\FPdim}(\alpha_s),\widetilde{\FPdim}(\alpha_t + \alpha_u)). 
\end{align*}
Moreover, since $\FPdim(r) = \FPdim(r^*)$, we have
\begin{align*}
\FPdim(B_R(r_1\cdot \alpha_s,r_2\cdot \alpha_t)) 
    &= \FPdim(r_2B_R(\alpha_s,\alpha_t)r_1^*) \\
    &= \FPdim(r_2)\FPdim(B_R(\alpha_s,\alpha_t))\FPdim(r_1),    
\end{align*}
which agrees with
\begin{align*}
B_\bbR(\widetilde{\FPdim}(r_1\cdot \alpha_s),\widetilde{\FPdim}(r_2 \cdot \alpha_t)) 
    &= B_\bbR(\FPdim(r_1)\alpha_s, \FPdim(r_2)\alpha_t) \\
    &= \FPdim(r_1)\FPdim(r_2)B_\bbR(\alpha_s,\alpha_t).
\end{align*}

However, we warn the reader that $\widetilde{\FPdim}$ need not be an injective group homomorphism, since $\FPdim$ itself need not be injective.
\begin{proposition} \label{prop:identifywithdualspace}
Let $R\Lambda_S$ be a geometric realisation of $(\bbW,S)$ over $R$.
Then the $\bbR$-linear map 
\begin{align*}
? \circ \widetilde{\FPdim}: \Hom_\bbR(\bbR\Lambda_S,\bbR) &\ra \Hom_R(R\Lambda_S,\bbR) \\
Z &\mapsto \widetilde{Z}:= Z\circ \widetilde{\FPdim}
\end{align*} 
is a $\bbW$-equivariant isomorphism.
\end{proposition} 
\begin{proof}
It is immediate that the linear map defined is an isomorphism of  $\bbR$-vector spaces.
Indeed, the dimensions (over $\bbR$) are equal since in both cases the morphisms are completely determined by the images of $\alpha_s$ for all $s \in S$, and the map is clearly surjective.
What remains is to show that the linear map is equivariant with respect to the $\bbW$-actions.

Recall from \eqref{eqn:bilinearformsFPdim} that the bilinear form $B_R(-,-)$ is identified with the bilinear form $B_\bbR(-,-)$ via $\FPdim$ and $\widetilde{\FPdim}$.
Since the two $\bbW$-actions are defined on each generator $s \in \bbW$ using the respective bilinear forms, and the $R$-module structure on $\bbR$ is defined precisely via $\FPdim$, it follows that the linear map is moreover $\bbW$-equivariant.
\end{proof}

As a corollary, we have the following, which says that every geometric $R$-realisation is in fact faithful.
\begin{corollary}\label{cor:faithfulfusionrealisation}
    Let $R\Lambda_S$ be a geometric $R$-realisation of $(\bbW,S)$.
    Then the contragredient $\bbW$-representation on $\Hom_R(R\Lambda_S,\bbR)$ is faithful, and the realisation $R\Lambda_S$ is also faithful.
\end{corollary}
\begin{proof}
    Using \cref{prop:identifywithdualspace}, we have that the $\bbW$-representation $\Hom_R(R\Lambda_S,\bbR)$ is faithful since the contragredient geometric representation $\Hom_\bbR(\bbR\Lambda_S,\bbR)$ is faithful.
    The faithfulness of the $\bbW$-representation $R\Lambda_S$ follows from the faithfulness on $\Hom_R(R\Lambda_S, \bbR)$.
\end{proof}


\subsection{Unfolding for the contragredient representation} \label{sec:unfoldcontragredientfusionrep}
Recall again that $R$ is a free $\bbZ$-module and that any $R$-linear morphism is automatically $\bbZ$-linear (i.e.\ a morphism of abelian groups).
This leads to the following immediate consequences:
\begin{enumerate}
    \item We have a $\bbR$-linear subspace inclusion
    \[
    \Hom_R(R\Lambda_S, \bbR) \subseteq \Hom_\bbZ(R\Lambda_S, \bbR).
    \]
    (This is $\Theta \subseteq \check{\Theta}$ in the Introduction.)
    \item The contragredient action of $\bbW$ on $\Hom_R(R\Lambda_S,\bbR)$ extends to the larger space $\Hom_\bbZ(R\Lambda_S, \bbR)$.
\end{enumerate}

The following (dual) proposition follows immediately from \cref{prop:unfoldedidentification}:
\begin{proposition} \label{prop:dualunfoldedidentification}
Let $R\Lambda_S$ be a geometric realisation of $(\bbW,S)$ over $R$, and let $\bbZ\Lambda_{\check{S}}$ be the unfolded realisation of the unfolded Coxeter system $(\check{\bbW}, \check{S})$ over $\bbZ$.
The $\bbR$-linear inclusion
\begin{align*}
    \Hom_R(R\Lambda_S, \bbR) \subseteq \Hom_\bbZ(R\Lambda_S, \bbR) &\xrightarrow{\cong} \Hom_\bbZ(\bbZ\Lambda_{\check{S}}, \bbR)
\end{align*}
is $\bbW$-equivariant with respect to the unfolding homomorphism $\varphi: \bbW \to \check{\bbW}$ (see \cref{defn:unfoldinghomomorphism}).
\end{proposition}

\begin{corollary} \label{cor:Coxgroupinjection}
    The unfolding homomorphism $\varphi: \bbW \rightarrow \check{\bbW}$
    is injective. 
\end{corollary}
\begin{proof}
    By \cref{cor:faithfulfusionrealisation} we have that the action of $\bbW$ on $\Hom_R(R\Lambda_S, \bbR)$ is faithful.
    Similarly, by \cref{lem:Zfaithful}, the action of $\check{\bbW}$ on $\Hom_\bbZ(R\Lambda_S,\bbR) \cong \Hom_\bbZ(\bbZ\Lambda_{\check{S}},\bbR)$ is faithful.
    The statement now follows from the fact that the (forgetful) map which sends $R$-linear automorphisms to $\bbZ$-linear automorphisms is injective.
\end{proof}

\begin{remark}\label{rem:unfoldinginjective}
    We expect the unfolding homomorphism $\varphi$ to be a map induced by a strong admissible partition of $\check{\Gamma}$, where injectivity is implied; see \cref{sec:partitionfoldingLCM} for more details.
\end{remark}

\subsection{Vinberg systems, hyperplane systems and hyperplane complements}
We first recall the definition of a Vinberg system; see e.g.\ \cite{paris2012kpi1}.
\begin{definition}
Let $V$ be a finite-dimensional real vector space.
Let $C \subset V$ be a closed polyhedral cone with nonempty interior, and let $C^\circ$ be its interior.
A \emph{wall} of $C^\circ$ is the support of a codimension one face of $C^\circ$, i.e. the hyperplane corresponding to that face.
Let $H_1, \dots, H_n$ be the walls of $C^\circ$.
Choose $S:= \{s_1, \dots, s_n\}$ to be a set of reflections in $V$, so that each $s_i$ fixes the hyperplane $H_i$. 
Note that there is no orthogonality condition in this definition -- there are infinitely many reflections that fix a given $H_i$, and we choose one $s_i$ among them.
Let $W$ denote the subgroup of $GL(V)$ generated by $S$.
If $w \cdot C^\circ \cap C^\circ = \emptyset \Leftrightarrow w = \id$, then $(C^\circ, S)$ is called a \emph{Vinberg pair}, and $C^\circ$ is called the \emph{fundamental chamber}.
The pair $(W,S)$ is also called a \emph{Vinberg system}.
\end{definition}

The following theorem of Vinberg is crucial:
\begin{theorem}[\cite{Vinberg_reflectiongroups}] 
    Let $(W,S)$ be a Vinberg system with fundamental chamber $C^\circ$.
    Set $T \coloneqq \bigcup_{w \in W} w \cdot C \subseteq V$.
    Then the following statements hold.
    \begin{enumerate}
        \item $T$ is a convex cone with non-empty interior.
        \item The interior $T^\circ$ of $T$ is invariant under the action of $W$, and the action of $W$ on $T^\circ$ is properly discontinuous.
        \item Let $x \in T$. Then $x \in T^\circ$ if and only if the stabiliser of $x$ is a finite subgroup of $W$.
        \item Let $x \in T^\circ$ be a point with non-trivial stabiliser. Then there exists a reflection $r \in \mathcal{R} \coloneqq \{ ws_iw^{-1} \in W \mid w\in W, s_i \in S\}$ such that $r\cdot x = x$.
    \end{enumerate}
    Moreover, $(W,S)$ is a Coxeter system and we call $T^\circ$ the Tits cone of $(W,S)$.
\end{theorem}

Given a Vinberg system $(W,S)$, for each reflection $r \in \mathcal{R} \subset W$, let $H_r \subset V$ denote the hyperplane fixed by $r$.
Then $\mathcal{A} \coloneqq \{ H_r \mid r \in \mathcal{R}\}$ is a \emph{hyperplane system} in $T^\circ$ -- this means that $\mathcal{A}$ is a (not necessarily finite) subset of hyperplanes in $V$ such that
\begin{enumerate}
	\item for all $H \in \mathcal{A}$, $H \cap T^\circ \neq \emptyset$; and 
	\item (locally finite) every point $x \in T^\circ$ has an open neighbourhood $U_x$ such that $U_x$ only intersects with finitely many hyperplanes in $\mathcal{A}$.
\end{enumerate}
\begin{definition} \label{defn:complexhyperplanecomplement}
Let $(W,S)$ be a Vinberg system with Tits cone $T^\circ$ and hyperplane system $\mathcal{A}$.
The (complexified) \emph{hyperplane complement} $\Omega_{\reg}$ associated to $(W,S)$ is defined as
\[
\Omega_{\reg} \coloneqq (T^\circ \times T^\circ) \setminus \cup_{H \in \mathcal{A}} (H \times H)
\]
\end{definition}

It follows from Vinberg's result above that the $W$-action on $\Omega_{\reg}$ is free and properly discontinuous.
Moreover, Van der Lek \cite{VdL_thesis} shows that
\[
\pi_1(\Omega_{\reg}/W) \cong \bbB(W,S)
\]
for any Vinberg system, where $\bbB(W,S)$ is the Artin--Tits group associated to the Coxeter system $(W,S)$ (see \cref{defn:Coxsystem}).

The following is a classical result by Tits and Vinberg, which shows that every Coxeter system $(\bbW,S)$ can be viewed as a Vinberg system using the contragredient (dual) representation associated to a geometric $\bbR$-realisation.
\begin{theorem}[\cite{Bourbaki_456,Vinberg_reflectiongroups}] \label{thm:Vinbergfromreal}
    Let $(\bbW,S)$ be a Coxeter system, $\bbR\Lambda_S$ a geometric $\bbR$-realisation and $\Hom_\bbR(\bbR\Lambda_S,\bbR)$ its corresponding contragredient representation.
    Let $C_\bbR := \{ Z \in \Hom_\bbR(\bbR\Lambda_S,\bbR) \mid Z(\alpha_s) \geq 0 \text{ for all } s \in \Gamma_0 \}$.
    The following statements hold:
    \begin{enumerate}
        \item $C_\bbR$ is a closed convex polyhedral cone with non-empty interior
        \[
        C_\bbR^\circ = \{ Z \in \Hom_\bbR(\bbR\Lambda_S,\bbR) \mid \forall s\in S : Z(\alpha_s) > 0 \}.
        \]
        \item The walls of $C_\bbR^\circ$ are given by $H_{\alpha_s} \coloneqq \{Z \in \Hom_\bbR(\bbR\Lambda_S,\bbR) \mid Z(\alpha_s) = 0\}$ for each $s \in S$.
        \item Each $s \in S$ fixes exactly $H_{\alpha_s}$.
        \item $w \cdot C_\bbR^\circ \cap C_\bbR^\circ \neq \emptyset$ if and only if $w = \id$.
    \end{enumerate}
    In particular, the contragredient representation of $\bbW$ makes $(\bbW,S)$ into a Vinberg system with fundamental chamber $C_\bbR^\circ$.
\end{theorem}

In the next subsection we will use the following useful lemma, which is probably well-known to the experts.
\begin{lemma}\label{lem:Titsconeconvexhull}
Let $T^\circ$ be the Tits cone of a Vinberg system $(W,S)$ with fundamental chamber $C^\circ$.
Then $T^\circ$ is the convex hull of $\bigcup_{w \in W} w\cdot C^\circ$.
\end{lemma}
\begin{proof}
It is clear that $\bigcup_{w \in W} w\cdot C^\circ \subseteq T^\circ$, since $T^\circ$ is the interior of $T=\bigcup_{w \in W} w\cdot C$ and $C^\circ$ is an open cone.
Since $T^\circ$ is convex, it also contains the convex hull of $\bigcup_{w \in W} w\cdot C^\circ$.

Now suppose $x \in T^\circ$.
The hyperplane system $\mathcal{A}$ is locally-finite in $T^\circ$, so there exists an open ball $U_x$ (in $T^\circ$) containing $x$ such that only finitely many hyperplanes $\{H_1,...,H_m\} \subseteq \mathcal{A}$ intersect non-trivially with $U_x$.
In particular, $U_x \setminus \left( \cup_{1 \leq i \leq m} H_i \right)$ is contained in $\bigcup_{w \in W} w\cdot C^\circ$.
The convex hull of $U_x \setminus \left( \cup_{1 \leq i \leq m} H_i \right)$ is $U_x$, which therefore contains $x$, as required.
\end{proof}

\subsection{Two Vinberg systems from one fusion realisation}
We will show that any geometric realisation $R\Lambda_S$ of $(\bbW,S)$ over a fusion ring $R$ in fact produces \emph{two} Vinberg systems: one for $(\bbW,S)$ and one for its corresponding unfolded Coxeter system $(\check{\bbW},\mathfrak{B}\times S)$ (see \cref{defn:unfoldedCox}).
\begin{proposition}\label{prop:twoVinbergsystems}
Let $R\Lambda_S$ be a $R$-geometric realisation of $(\bbW,S)$, and let $(\check{\bbW}, \mathfrak{B}\times S)$ be the corresponding unfolded Coxeter system.
Define 
\begin{itemize}
    \item $\check{C} \coloneqq \{Z \in \Hom_\bbZ(R\Lambda_S,\bbR) \mid \forall b \in \mathfrak{B}, \forall s \in S:Z(b\alpha_{s}) \geq 0 \}$; and 
    \item $C \coloneqq \{Z \in \Hom_R(R\Lambda_S, \bbR) \mid  \forall s \in S:Z(\alpha_{s}) \geq 0 \}$.
\end{itemize}
The following statements hold:
\begin{enumerate}
    \item Both $\check{C}$ and $C$ are closed convex polyhedral cones, with interiors 
    \begin{itemize}
    \item $\check{C}^\circ =  \{Z \in \Hom_\bbZ(R\Lambda_S,\bbR) \mid \forall b \in \mathfrak{B}, \forall s \in S:Z(b\alpha_{s}) >0 \}$; and 
    \item $C^\circ = \{Z \in \Hom_R(R\Lambda_S, \bbR) \mid \forall s \in S:Z(\alpha_{s}) > 0 \}$
    \end{itemize}
    respectively.
    \item The walls of $\check{C}$ and $C$ are given respectively by 
    \begin{itemize}
    \item $\check{H}_{b\alpha_s} \coloneqq \{ Z \in \Hom_\bbZ(R\Lambda_S, \bbR) \mid Z(b\alpha_s) = 0\}$ for each $s \in S, b \in \mathfrak{B}$; and 
    \item $H_{\alpha_{s}} \coloneqq \{ Z \in \Hom_R(R\Lambda_S, \bbR) \mid Z(\alpha_s) = 0\}$ for each $s \in S$.
    \end{itemize}
    \item Each $(b,s) \in \mathfrak{B} \times S \subset \check{\bbW}$ fixes exactly $\check{H}_{b\cdot \alpha_s}$ in $\Hom_\bbZ(R\Lambda_S,\bbR)$, and each $s \in S$ fixes exactly $H_{\alpha_s}$ in $\Hom_R(R\Lambda_S,\bbR)$.
    \item For all $w \in \check{\bbW}$, we have $w \cdot \check{C}^\circ \cap \check{C}^\circ \neq \emptyset$ if and only if $w = \id$; similarly for all $w \in \bbW$, we have $w \cdot C^\circ \cap C^\circ \neq \emptyset$ if and only if $w = \id$.
\end{enumerate}
In particular, the Vinberg pairs $(\check{C}^\circ,\mathfrak{B}\times S)$ and $(C^\circ,S)$ make $(\check{\bbW},\mathfrak{B}\times S)$ and $(\bbW,S)$ into Vinberg systems respectively.
\end{proposition}
\begin{proof}
These statements follow directly from the relation of these spaces with the ones constructed from geometric $\bbR$-realisations.

More precisely, let $\bbR\Lambda_{\check{S}}$ be the geometric $\bbR$-realisation of $(\check{\bbW},\check{S})$ using the same Cartan matrix for the $\bbZ$-realisation $\bbZ\Lambda_{\check{S}}$.
On the other hand, let $\bbR\Lambda_S$ be the geometric $\bbR$-realisation of $(\bbW,S)$ given by the Cartan matrix $(\FPdim(r_{s,t}))_{s,t\in S}$.
We have identifications (first one is from \cref{prop:dualunfoldedidentification}, second is obvious):
\[
\Hom_\bbZ(R\Lambda_S,\bbR) \cong \Hom_\bbZ(\bbZ\Lambda_{\check{S}}, \bbR) \cong \Hom_\bbR(\bbR\Lambda_{\check{S}}, \bbR)
\]
as $\check{\bbW}$-representations over $\bbR$; on the other hand the identification
\[
\Hom_R(R\Lambda_S,\bbR) \cong \Hom_\bbR(\bbR\Lambda_S,\bbR)
\]
as $\bbW$-representations over $\bbR$ is shown in \cref{prop:identifywithdualspace}.
We leave it to the reader to check that under the identification above, the polyhedral cones and the walls defined in this proposition agree with those obtained from the contragradient $\bbR$-representation as in \cref{thm:Vinbergfromreal}.
\end{proof}

The objects associated to the two Vinberg systems above interact as follows.
\begin{proposition} \label{prop:interaction}
    \begin{enumerate}
        \item \label{item:chamberintersect} $C = \check{C} \cap \Hom_R(R\Lambda_S,\bbR)$ and $C^\circ = \check{C}^\circ \cap \Hom_R(R\Lambda_S,\bbR)$.
        \item \label{item:wallintersect} For all $s \in S, b \in \mathfrak{B}$, $H_{\alpha_s}=\check{H}_{b\alpha_s} \cap \Hom_R(R\Lambda_S,\bbR)$.
        Moreover, $H_{\alpha_s} = \Hom_R(R\Lambda_S,\bbR) \cap \left(\bigcap_{b \in \mathfrak{B}} \check{H}_{b\alpha_s} \right)$.
        \item \label{item:restrictedaction} For all $s \in S \subset \bbW$, the action of $\varphi(s) = \prod_{b\in \mathfrak{B}} (b,s) \in \check{\bbW}$ restricted to $\Hom_R(R\Lambda_S,\bbR)$ agrees with the action of $s$, and therefore fixes the same hyperplane $H_{\alpha_s}$.
    \end{enumerate}
\end{proposition}
\begin{proof}
    Note that if $Z \in \Hom_R(R\Lambda_S,\bbR)$, then $Z(b\alpha_s) = \FPdim(b)Z(\alpha_s)$.
    By \cref{prop:FPdimproperties}, for all $b \in \mathfrak{B}$ we get $\FPdim(b) \geq 1 > 0$, so statements in \ref{item:chamberintersect} and \ref{item:wallintersect} hold.
    Statement \ref{item:restrictedaction} is a consequence of \cref{prop:dualunfoldedidentification}.
\end{proof}

\subsection{Main results}
Following \cref{prop:twoVinbergsystems}, for the rest of this section we shall view both $(\check{\bbW}, \check{S})$ and $(\bbW, S)$ as Vinberg systems.
Associated to the Vinberg system $(\check{\bbW}, \mathfrak{B} \times S)$, we denote
\begin{itemize}
	\item the Tits cone by $\check{T}^\circ \subseteq \Hom_\bbZ(R\Lambda,\bbR)$;
	\item the hyperplane system by $\check{\mathcal{A}}$; and 
	\item the hyperplane complement by $\check{\Omega}_{\reg} \coloneqq (\check{T}^\circ \times \check{T}^\circ) \setminus \cup_{\check{H} \in \check{\mathcal{A}}} \left(\check{H} \times \check{H} \right)$.
\end{itemize}
Similarly, associated to the Vinberg system $(\bbW, S)$, we denote
\begin{itemize}
	\item the Tits cone by $T^\circ \subseteq \Hom_R(R\Lambda,\bbR)$;
	\item the hyperplane system by $\mathcal{A}$; and 
	\item the hyperplane complement by $\Omega_{\reg} \coloneqq (T^\circ \times T^\circ) \setminus \cup_{H \in \mathcal{A}} \left(H \times H \right)$.
\end{itemize}

Our first aim is to relate the two hyperplane systems $\mathcal{A}$ and $\mathcal{\check{A}}$.
Note that \cref{prop:interaction} implies that $\check{H} \cap \Hom_R(R\Lambda_S,\bbR)$ is a hyperplane in $\Hom_R(R\Lambda_S, \bbR)$ for all $\check{H}\in \check{\mathcal{A}}$.
Moreover,
\begin{equation} \label{eqn:hyperplaneinclu}
\mathcal{A} \subseteq \{ \check{H} \cap \Hom_R(R\Lambda_S,\bbR) \mid \check{H} \in \check{\mathcal{A}} \}.
\end{equation}
\begin{example}
    Let $(\bbW,S)$ be the Coxeter system associated to $\Gamma = I_2(5)$ and consider its geometric $R^e_M$-realisation as in \cref{eg:pentagonunfolding}, so that $(\check{\bbW},\check{S})$ is of type $A_4$.
    Here the inclusion in \eqref{eqn:hyperplaneinclu} is in fact an equality, with the 10 hyperplanes of $A_4$ intersecting $T^\circ = \Hom_{R^e_M}(R^e_M\Lambda_S,\bbR)$ to give the 5 reflection lines of the pentagon as in \cref{fig:pentagonunfold}.
\end{example}

The following example shows that the inclusion in \eqref{eqn:hyperplaneinclu} can be strict.
\begin{example} \label{eg:hyperplaneoutsideTitscone}
Let $(\bbW,S)$ be the Coxeter system associated to the Coxeter graph of affine 
$\widetilde{A}_1 = 
\begin{tikzcd}[column sep=small]
s \ar[r, no head, "\infty"] & t
\end{tikzcd}
$ 
type and let $R = R(S_3)$ be the representation ring of $S_3$ with basis $\mathfrak{B} = \{1, V, \sgn\}$.
Consider the geometric $R$-realisation $R\Lambda_S$ with Cartan matrix given by
\[
\begin{bmatrix}
	2  & -V \\
	-V &  2 
\end{bmatrix}.
\]
The corresponding unfolded Coxeter system is of affine $\widetilde{D}_5$ type:
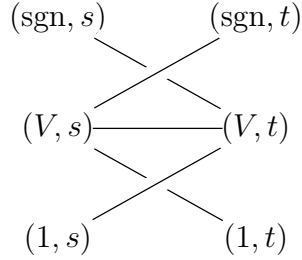
\[\begin{tikzcd}[every arrow/.append style = {shorten <= -.5em, shorten >= -.5em}]
    (\sgn,s) & (\sgn,t) \\
	(V,s)    & (V,t) \\
	(1,s)    & (1,t)
	\arrow[no head, from=1-1, to=2-2]
	\arrow[no head, from=2-1, to=1-2]
	\arrow[no head, from=2-1, to=2-2]
	\arrow[no head, from=2-1, to=3-2]
	\arrow[no head, from=3-1, to=2-2]
\end{tikzcd}\]
Then 
\[
\check{H} \coloneqq
	\{
	Z \in \Hom_\bbZ(R\Lambda_S, \bbR) \mid Z(\alpha_{s} + V\alpha_{t} + V\alpha_{s} + \sgn \alpha_{t}) = 0
	\}
\]
is a hyperplane in $\check{\mathcal{A}}$.
If in addition $Z \in \Hom_R(R\Lambda_S,\bbR)$, note that 
\[
Z(\alpha_{s} + V\alpha_{t} + V\alpha_{s} + \sgn \alpha_{t}) = Z(\alpha_s)+2Z(\alpha_t) + 2Z(\alpha_s) + Z(\alpha_t)
\]
(recall $\FPdim(\sgn)=1, \FPdim(V)=2$).
As such,
\[
\check{H} \cap \Hom_R(R\Lambda_S,\bbR)
	= \{ Z \in \Hom_R(R\Lambda_S,\bbR) \mid Z(\alpha_s) + Z(\alpha_t) = 0\},
\]
and the RHS is \emph{not} a hyperplane in $\mathcal{A}$ (in fact, it is the hyperplane associated to the imaginary root $\alpha_s + \alpha_t$).
\end{example}

\begin{remark}\label{rmk:distinct_realisation}
    We quickly note that the unfolded Coxeter system $(\check{\bbW}, \check{S})$ depends on the realisation of $(\bbW,S)$.
    For example, instead of the $R$-realisation of $(\bbW,S)$ given in \cref{eg:hyperplaneoutsideTitscone}, one can consider the $R$-realisation (over the same $R$) with Cartan matrix $\begin{bsmallmatrix} 2 &-(2\cdot 1) \\ -(2\cdot 1) & 2 \end{bsmallmatrix}$.
    The resulting unfolded Coxeter system is associated to 3 copies of affine $\widetilde{A}_1$.
\end{remark}

Although the above example presents a hyperplane $\check{H}$ such that $\check{H} \cap \Hom_R(R\Lambda_S, \bbR)$ is not a hyperplane in $\mathcal{A}$, notice that it lives outside of the Tits cone $T^\circ$.
Our main theorem is to show that this is the only possibility.
\begin{theorem}\label{thm:hyperplanesystem}
	The hyperplane system $\mathcal{A}$ can be identified as follows:
	\[
	\mathcal A = \{ \check H_r \cap \mathrm{Hom}_R (R \Lambda_S, \mathbb R) \mid \check H_r \in \check{\mathcal A} \text{ and } \check H_r \cap T^\circ \neq \emptyset\}
	\]
	In other words, the hyperplane system in $T^\circ$ induced from $\check{\mathcal{A}}$ through the intersection with $\Hom_R(R\Lambda_S,\bbR)$ agrees with $\mathcal{A}$ (noting that hyperplanes of a hyperplane system in $T^\circ$ must intersect $T^\circ$ non-trivially by definition).
\end{theorem}
We start with the following lemma, which is slightly weaker than the theorem.
\begin{lemma}
    Suppose $\check{H}\cap T \neq \emptyset$. Then for all $x \in \check{H}\cap T$, $x \in H_x$ for some $H_x \in \mathcal{A}$.
\end{lemma}
\begin{proof}
    Let $x \in \check{H}\cap T \neq \emptyset$.
    Then there exists $w \in \bbW$ such that $w\cdot x \in C \cap T$, and so $\varphi(w)\cdot x \in  \check{C}$.
    Recall that the only points in $\check{C}$ with non-trivial $\check{\bbW}$-stabilisers are points on the walls.
    It follows that $\varphi(w)\cdot x \in \check{H}_{b_i\alpha_s}$ for some $b_i \in \mathfrak{B}, s \in S$.
    Since $\check{H}_{b_i\alpha_s} \cap T = H_{\alpha_s} \cap T$, we have that $x \in H_x \coloneqq w^{-1}\cdot H_{\alpha_s}$ as required.
\end{proof}
\begin{proof}[Proof of \cref{thm:hyperplanesystem}]
Suppose $\check{H} \cap T^\circ \neq \emptyset$.
By the lemma above, there exists $H \in \mathcal{A}$ such that $H \cap \left( \check{H} \cap T^\circ \right) \neq \emptyset$.
We claim that $H$ can moreover be chosen so that $H = \check{H}\cap \Hom_R(R\Lambda_S,\bbR)$.

Suppose to the contrary that $H \neq \check{H} \cap \Hom_R(R\Lambda_S,\bbR)$ for all $H \in \mathcal{A}$.
Since $T^\circ \subseteq \Hom_R(R\Lambda_S,\bbR)$ is an open convex cone and $\check{H} \cap \Hom_R(R\Lambda_S,\bbR)$ is a hyperplane, $\check{H} \cap T^\circ$ is a convex cone of codimension 1 (in $\Hom_R(R\Lambda_S,\bbR)$).
But for all $x \in \check{H}\cap T^\circ$ we have some $H_x \in \mathcal{A}$ containing $x$, so this would imply that any open subset of $\check{H} \cap T^\circ$ has to intersect with infinitely many hyperplanes $H \in \mathcal{A}$.
In particular, there exists $x \in \check{H}\cap T^\circ \subset T^\circ$ with the property that any open set $U_x \ni x$ will intersect with infinitely many hyperplanes in $\mathcal{A}$, contradicting the locally-finite property of $\mathcal{A}$ in $T^\circ$.
\end{proof}

The following corollary is an immediate consequence of the theorem above.
\begin{corollary}
Suppose $(\bbW,S)$ is a finite Coxeter system. 
Then the inclusion in \eqref{eqn:hyperplaneinclu} is an equality.
\end{corollary}

Next, we prove our main theorem on embeddings of hyperplane complements:
\begin{theorem}\label{thm:hyperplanecompembedding}
    The natural embedding $\Hom_R(R\Lambda_S,\bbR) \subseteq \Hom_\bbZ(R\Lambda_S,\bbR)$ induces an embedding of hyperplane complements $\Omega_{\reg} \subseteq \check{\Omega}_{\reg}$, which moreover is $\bbW$-equivariant with respect to $\varphi: \bbW \to \check{\bbW}$.
    The homomorphism between fundamental groups $\tilde{\varphi}:\pi_1(\Omega_{\reg}/\bbW) \to \pi_1(\check{\Omega}_{\reg}/\check{\bbW})$ induced by the embedding is defined by sending each standard generator $\sigma_s \mapsto \prod_{b \in \mathfrak{B}} \sigma_{(b,s)}$.
\end{theorem}
\begin{proof}
We first show that $T^\circ$ indeed sits inside $\check{T}^\circ$.
Indeed, by \cref{prop:interaction}, $C^\circ \subseteq \check{C}^\circ$ and
\[
\bigcup_{w \in \bbW} w \cdot C^\circ \subseteq \bigcup_{w \in \bbW} \varphi(w) \cdot \check{C}^\circ \subseteq \bigcup_{\check{w} \in \check{\bbW}} \check{w} \cdot \check{C}^\circ.
\]
Taking the convex hulls of the sets above, we get $T^\circ \subseteq \check{T}^\circ$ by \cref{lem:Titsconeconvexhull} as desired.
It now follows from \cref{thm:hyperplanesystem} that $\Omega_{\reg} \subseteq \check{\Omega}_{\reg}$.
Moreover, the fact that the embedding $\Omega_{\reg} \subseteq \check{\Omega}_{\reg}$ is $\bbW$-equivariant with respect to $\varphi$ is a consequence of \ref{prop:interaction}.

Now, we must show that the induced group homomorphism $\tilde{\varphi}: \pi_1 (\Omega_{\reg} / \bbW) \to \pi_1 (\check{\Omega}_{\reg} / \check{\bbW})$ sends each standard generator $\sigma_s$ to $\prod_{b \in \mathfrak{B}} \sigma_{(b,s)}$.
We denote by $\pi : \Omega_{\reg} \to \Omega_{\reg} / \bbW$ and by $\check{\pi} : \check{\Omega}_{\reg} \to \check{\Omega}_{\reg} / \check{\bbW}$ the canonical quotients, we choose a point $x_0 \in C^\circ \subseteq \check C^\circ$, and we take $z_0 \coloneqq \pi (x_0,x_0)$ as base point for $\Omega_{\reg} / \bbW$ and $\check{z}_0 \coloneqq \check \pi (x_0,x_0)$ as base point for $\check{\Omega}_{\reg} / \check \bbW$.

\smallskip\noindent
For $w_1, w_2 \in \bbW$ and $s \in S$ we define the paths $\alpha_s (w_1,w_2)$ and $\beta_s (w_1,w_2)$ in $\Omega_{\reg}$ by:
\begin{gather*}
\alpha_s (w_1, w_2) (t) = ((1-t) (w_1 \cdot x_0) + t (w_1s \cdot x_0), w_2 \cdot x_0)\,, \\
\beta_s (w_1, w_2) (t) = (w_1 \cdot x_0, (1-t) (w_2 \cdot x_0) + t (w_2 s \cdot x_0))\,.
\end{gather*}
Note that the loop $\gamma_s = \pi \circ (\alpha_s (w,w) \beta_s (ws,w))$ in $\Omega_{\reg} / \bbW$ does not depend on $w$ and represents the generator $\sigma_s$.
On the other hand, for $\check w_1, \check w_2 \in \check \bbW$ and $(b,s) \in \mathfrak{B} \times S$, we define the paths $\check \alpha_{(b,s)} (\check w_1,\check w_2)$ and $\check \beta_{(b,s)} (\check w_1,\check w_2)$ in $\check \Omega_{\reg}$ by:
\begin{gather*}
\check \alpha_{(b,s)} (\check w_1, \check w_2) (t) = ((1-t) (\check w_1 \cdot x_0) + t (\check w_1 (b,s) \cdot x_0), \check w_2 \cdot x_0)\,,\\
\check \beta_{(b,s)} (\check w_1, \check w_2) (t) = (\check w_1 \cdot x_0, (1-t) (\check w_2 \cdot x_0) + t (\check w_2 (b,s) \cdot x_0))\,.
\end{gather*}
Note that the loop $\check \gamma_{(b,s)} = \check \pi \circ (\check \alpha_{(b,s)} (\check w,\check w) \check \beta_{(b,s)} (\check w (b,s),\check w))$ in $\check \Omega_{\reg} / \check \bbW$ does not depend on $\check w$ and represents $\sigma_{(b,s)}$.

\smallskip\noindent
We order $\mathfrak{B} = \{b_1, \dots, b_q\}$ and we set:
\begin{gather*}
\theta =
\check \alpha_{(b_1,s)} (1,1) \check \beta_{(b_1,s)} ((b_1,s), 1)
\check \alpha_{(b_2,s)} ((b_1,s), (b_1,s)) \check \beta_{(b_2,s)} ((b_1,s)(b_2,s), (b_1,s)) \cdots\\
\check \alpha_{(b_q,s)} ((b_1,s) \cdots (b_{q-1},s), (b_1,s) \cdots (b_{q-1},s)) \check \beta_{(b_q,s)} ((b_1,s) \cdots (b_{q-1},s)(b_q,s),(b_1,s) \cdots (b_{q-1},s))\,.
\end{gather*}
Then, by the above, $\prod_{b \in \mathfrak{B}} \sigma_{(b,s)}$ is represented by the loop $\check \pi \circ \theta$ in $\check \Omega_{\reg} / \check \bbW$.

Now, we fix $s \in S$.
For $b_i \in \mathfrak{B}$ and $\check w \in \langle (b_1,s), \dots, (b_q,s) \rangle$, the only hyperplane of $\check{\mathcal{A}}$ that the segment $[0,1] \to \Hom_\bbZ (R \Lambda_S, \bbR)$, $t \mapsto (1-t) (\check w \cdot x_0) + t (\check w (b_i,s) \cdot x_0)$, intersects is $\check H_{b_i \cdot \alpha_s}$.
It follows that, for $i \neq j$ and $\check w_1, \check w_2 \in \langle (b_1,s), \dots, (b_q,s) \rangle$, we have an embedding $[0,1] \times [0,1] \to \check \Omega_{\reg}$ defined by
\[
(t_1,t_2) \mapsto ((1-t_1) (\check w_1 \cdot x_0) + t_1 (\check w_1 (b_i,s) \cdot x_0), (1-t_2) (\check w_2 \cdot x_0) + t_2 (\check w_2 (b_j,s) \cdot x_0)\,.
\]
The boundary of this square is
\[
\check \alpha_{(b_i,s)} (\check w_1, \check w_2)
\check \beta_{(b_j,s)} (\check w_1 (b_i,s), \check w_2)
\check \alpha_{(b_i,s)} (\check w_1, \check w_2 (b_j,s))^{-1}
\check \beta_{(b_j,s)} (\check w_1, \check w_2)^{-1}\,,
\]
hence $\check \alpha_{(b_i,s)} (\check w_1, \check w_2) \check \beta_{(b_j,s)} (\check w_1 (b_i,s), \check w_2)$ is homotopic to $\check \beta_{(b_j,s)} (\check w_1, \check w_2) \check \alpha_{(b_i,s)} (\check w_1, \check w_2 (b_j,s))$ in $\check \Omega_{\reg}$ with respect to the endpoints.
Set
\begin{gather*}
\theta'_a =
\check \alpha_{(b_1,s)} (1,1)
\check \alpha_{(b_2,s)} ((b_1,s),1) \cdots
\check \alpha_{(b_q,s)} ((b_1,s) \cdots (b_{q-1},s),1)\,,\\
\theta'_b =
\check \beta_{(b_1,s)}( (b_1,s) \cdots (b_q,s),1 )
\check \beta_{(b_2,s)}( (b_1,s) \cdots (b_q,s), (b_1,s) )
\cdots\\
\quad\quad\check \beta_{(b_q,s)}( (b_1,s) \cdots (b_q,s),(b_1,s) \cdots (b_{q-1},s))\,.
\end{gather*}
Then, by iterating the above homotopy, we get that $\theta$ is homotopic to $\theta'_a \theta'_b$ in $\check \Omega_{\reg}$ with respect to the endpoints.

\smallskip\noindent
$\theta'_a$ and $\alpha_s (1,1)$ are two paths in $\check T^o \times \{x_0\} \subseteq \check \Omega_{\reg}$ whose endpoints are $(x_0, x_0)$ and $((b_1,s) \cdots (b_q,s) \cdot x_0, x_0) = (s \cdot x_0, x_0)$.
Since $\check T^o \times \{x_0\}$ is convex, these two paths are homotopic relative to the endpoints.
Similarly, $\theta'_b$ and $\beta_s (s,1)$ are homotopic relative to the endpoints. 
So, $\check \pi \circ \theta$ is homotopic to $\gamma_s$ in $\check \Omega_{\reg} / \check \bbW$, hence the induced group homomorphism $\pi_1 (\Omega_{\reg} / \bbW) \to \pi_1 (\check \Omega_{\reg} / \check \bbW)$ sends $\sigma_s$ to $\prod_{b \in \mathfrak{B}} \sigma_{(b,s)}$.
\end{proof}

\begin{remark}
    The fact that $\tilde{\varphi}:\pi_1(\Omega_{\reg}/\bbW) \to \pi_1(\check{\Omega}_{\reg}/\check{\bbW})$ sends each standard generator $\sigma_s \mapsto \prod_{b \in \mathfrak{B}} \sigma_{(b,s)}$ can also be proven using the notion of (direct) pregalleries in \cite{VdL_thesis}; the details are left to the interested reader.
\end{remark}

Finally, we also deduce finiteness relation between the two Coxeter systems $(\check{\bbW},\check{S})$ and $(\bbW,S)$.
Given a subset $J \subseteq S$ (resp.\ $\subseteq \check{S}$), let $\bbW_J$ (resp.\ $\check{\bbW}_J$) denote the standard parabolic subgroup generated by $J$.
\begin{proposition}\label{prop:finiteiffunfoldedfinite}
Let $J \subseteq S$, so that $\mathfrak{B}\times J \subseteq \check{S}$.
The Coxeter group $\bbW_J$ is finite if and only if $\check{\bbW}_{\mathfrak{B}\times J}$ is also finite. 
In particular, $\bbW$ is finite if and only if $\check{\bbW}$ is finite.
\end{proposition}
\begin{proof}
We claim that for any $x \in T$, the stabiliser $\bbW_x$ of $x$ in $\bbW$ is finite if and only if the stabiliser $\check{\bbW}_x$ of $x$ in $\check{\bbW}$ is finite.
In the proof of \cref{thm:hyperplanecompembedding} above, we have shown that $T^\circ \subseteq \check{T}^\circ$.
So if $x \in T$ has finite stabiliser $\bbW_x$ in $\bbW$, then it is a point in $T^\circ$, hence a point in $\check{T}^\circ$, which implies that the stabiliser $\check{\bbW}_x$ in $\check{\bbW}$ is also finite.
On the other hand, if $\check{\bbW}_x$ is finite, then $\bbW_x$ is also finite since the injective group homomorphism $\varphi: \bbW \to \check{\bbW}$ sends $\bbW_x$ to $\check{\bbW}_x$ by \cref{prop:interaction}.

It is now sufficient to find an $x \in C \subseteq T$ with stabilisers $\bbW_x = \bbW_J$ and $\check{\bbW}_x = \check{\bbW}_{\mathfrak{B}\times J}$.
Given $J \subseteq S$, choose a point $x_J \in C$ with stabiliser in $\bbW$ given by $\bbW_J$; explicitly,
\[
x_J \in \{ Z \in C \mid \forall s \in J, Z(\alpha_s) = 0 \text{ and } \forall s \in S \setminus J, Z(\alpha_s) > 0 \}.
\]
Viewing $x_J \in \check{C}$, it follows that the stabiliser of $x_J$ in $\check{\bbW}$ is given by $\check{\bbW}_{\mathfrak{B}\times J}$, since any $x_J$ as above satisfies
\[
x_J \in \{ \check{Z} \in \check{C} \mid \forall b \in \mathfrak{B}, \forall s \in J: \check{Z}(b\alpha_s) = 0; \text{ and } \forall b \in \mathfrak{B}, \forall s \in S \setminus J: \check{Z}(b\alpha_s) > 0 \},
\]
using the fact that $x_J(b\alpha_s) = \FPdim(b)x_J(\alpha_s)$, and that $\FPdim(b) \geq 1 > 0$.
%
\end{proof}
\begin{remark}
    As mentioned in \ref{rem:unfoldinginjective}, the unfolding homomorphism $\varphi$ is expected to be induced by a strong admissible partition of $\check{\Gamma}$, where \cref{prop:finiteiffunfoldedfinite} is also implied; see \cref{sec:partitionfoldingLCM} for more details.
\end{remark}

\section{Strong admissible partitions, foldings and LCM homomorphisms}\label{sec:partitionfoldingLCM}
\subsection{Strong admissible partitions and foldings}
We recall here the definitions of strong admissible partitions (also called Lusztig partitions) and foldings.
\begin{definition}
Let $\check{\Gamma}$ be a Coxeter graph associated to a Coxeter system $(\check{\bbW},\check{S})$.
A \emph{strong admissible partition} of $\check{\Gamma}$ is a surjection $\pi: \check{S} \to S$ onto a finite set $S$, satisfying:
\begin{enumerate}
\item $\pi(\check{s}) = \pi(\check{s}') \implies $ $\check{s}$ and $\check{s}'$ are not connected by an edge in $\check{\Gamma}$; and  
\item given any two distinct elements $s \neq t \in S$, the full Coxeter subgraph of $\check{\Gamma}$ generated by $\pi^{-1}(s) \cup \pi^{-1}(t)$ is a disjoint union of irreducible Coxeter graphs with all having the same Coxeter number (the Coxeter number is $\infty$ if the irreducible Coxeter group is infinite); we denote this Coxeter number by $m_{s,t}$.
\end{enumerate}
The Coxeter matrix $(m_{s,t})_{s,t \in S}$ defines a Coxeter system $(\bbW,S)$, the Coxeter graph of which we denote by $\Gamma$.
Motivated by \cite{Crisp99,Castella_admissible}, the simplicial map of graphs $\check{\Gamma} \to \Gamma$ induced by $\pi$ is called a \emph{strong admissible folding}.
\end{definition}

\begin{remark}
The above is a special case of an \emph{admissible folding}; see \cite{Muhlherr_CoxinCox,Castella_admissible} for the definition.
\end{remark}

The relevance of strong admissible partitions and foldings to Coxeter groups and Artin--Tits groups are as follows.
\begin{definition}
Let $\pi$ be a strong admissible partition of $\check{\Gamma}$ and let $\check{\Gamma} \to \Gamma$ be its associated folding.
The map sending $s \in S$ to $\prod_{\check{s} \in \pi^{-1}(s)} \check{s} \in \check{\bbW}$ induces a well-defined group homomorphism $\varphi_\pi: \bbW \to \check{\bbW}$ \cite{Muhlherr_CoxinCox}, which we call a \emph{strong admissible homomorphism} between the Coxeter groups.

Similarly, let $\bbB$ and $\check{\bbB}$ denote the Artin--Tits groups associated to $(\bbW,S)$ and $(\check{\bbW},\check{S})$ respectively.
The map sending $\sigma_s \in \bbB$ to $\prod_{\check{s} \in \pi^{-1}(s)} \sigma_{\check{s}} \in \check{\bbB}$ induces a well-defined group homomorphism $\widetilde{\varphi}_\pi: \bbB \to \check{\bbB}$ between the Artin--Tits groups, which we call the \emph{strong admissible homomorphism} associated to the strong admissible folding $\check{\Gamma} \to \Gamma$.
\end{definition}
Note that strong admissible homomorphisms are special cases of LCM homomorphisms studied in \cite{Crisp99,Godelle_FCfolding}.

The following results on strong admissible homomorphisms can be found in \cite{Muhlherr_CoxinCox, Dyer_rootII}; the result on Artin--Tits monoids can be found in \cite{Castella_admissible}.
See also \cite{EH_Coxembedding}.
\begin{proposition} \label{prop:foldingresults}
\begin{enumerate}
\item A strong admissible homomorphism $\varphi_\pi: \bbW \to \check{\bbW}$ is injective.
\item An strong admissible homomorphism $\widetilde{\varphi}_\pi: \bbB \to \check{\bbB}$ associated to a strong admissible folding is injective on the respective Artin--Tits monoids.
\item Suppose $(\bbW,S)$ is irreducible. Then $\bbW$ and each irreducible component of $\check{\bbW}$ all have the same Coxeter number (including $\infty$).
In particular, $\bbW$ is finite if and only if $\check{\bbW}$ is finite.
\end{enumerate}
\end{proposition}
We note that $\widetilde{\varphi}_\pi: \bbB \to \check{\bbB}$ is also conjecturally injective (on the whole Artin--Tits group), where it is known for spherical type \cite{Crisp99} and FC type \cite{Godelle_FCfolding} Artin--Tits groups.

\subsection{Folding vs.\ unfolding}
Recall that in \cref{defn:unfoldedCox}, we defined an unfolded Coxeter system $(\check{\bbW},\check{S})$ from a Coxeter system $(\bbW,S)$ equipped with a geometric realisation over a fusion ring $R$.
The following result in \cite{EH_fusionquivers} justifies our naming convention:
\begin{theorem}[\protect{\cite[see Theorem 5.27 and proof of Corollary 6.9]{EH_fusionquivers}}]\label{thm:unfoldisLusztig}
Fix $(\bbW,S)$ to be a Coxeter system.
Let $R\Lambda_S$ be a geometric $R$-realisation of $(\bbW,S)$ for some fusion ring $R$ and let $(\check{\bbW}, \check{S})$ be the corresponding unfolded Coxeter system.
Suppose $R$ is categorifiable.
Then the partition $\pi: \check{S} \to S$ sending $(b,s) \mapsto s$ is a strong admissible partition.
\end{theorem}
In particular, as noted in \cref{eg:unfoldedsystems}, the $R_M$-realisation in \cref{eg:Verlinderealisation} corresponds to strong admissible foldings from type $A$'s, whereas $R^\infty_M$-realisation in \cref{eg:modifiedrealisation} also involves folding from $\widetilde{D}_5$ on the $\infty$-labelled edge.
Note that the latter is essentially a variant of the one considered in \cite{Paris_monoidinject}, which instead uses $\widetilde{A}_3$ in place of $\widetilde{D}_5$. 

Combining the theorem above with \cref{thm:hyperplanecompembedding}, we obtain the following result.
\begin{corollary}
Suppose $R$ is categorifiable.
Then the homomorphism of fundamental groups in \cref{thm:hyperplanecompembedding} is a strong admissible homomorphism associated to a strong admissible folding.
In other words, the map $\Omega_{\reg}/\bbW \to \check{\Omega}_{\reg}/\check{\bbW}$, induced from the embedding $\Omega_{\reg} \subseteq \check{\Omega}_{\reg}$ of the corresponding regular orbit spaces, is a topological realisation of the corresponding strong admissible homomorphism.
\end{corollary}

We conjecture that the assumption of $R$ being categorifiable is superfluous:
\begin{conjecture}
\cref{thm:unfoldisLusztig} holds for all fusion rings $R$.
In particular, the homomorphism of fundamental groups in \cref{thm:hyperplanecompembedding} is always a strong admissible homomorphism associated to a strong admissible folding.
\end{conjecture}
Note that by \cref{prop:foldingresults}, the conjecture above would also imply the group theoretical results stated in \cref{cor:Coxgroupinjection} and \cref{prop:finiteiffunfoldedfinite}.

\printbibliography

\end{document}